\documentclass[british,english]{amsart}
\usepackage[T1]{fontenc}
\usepackage[latin9]{inputenc}
\usepackage{xcolor}
\usepackage{array}
\usepackage{verbatim}
\usepackage{booktabs}
\usepackage{mathtools}
\usepackage{amsbsy}
\usepackage{amstext}
\usepackage{amsthm}
\usepackage{amssymb}
\usepackage{stmaryrd}
\usepackage{setspace}

\makeatletter

\providecommand{\tabularnewline}{\\}

\numberwithin{equation}{section}
\theoremstyle{plain}
\newtheorem{thm}{\protect\theoremname}[section]
  \theoremstyle{definition}
  \newtheorem{defn}[thm]{\protect\definitionname}
  \theoremstyle{definition}
  \newtheorem{example}[thm]{\protect\examplename}
  \theoremstyle{plain}
  \newtheorem{lem}[thm]{\protect\lemmaname}
  \theoremstyle{plain}
  \newtheorem{prop}[thm]{\protect\propositionname}
  \theoremstyle{remark}
  \newtheorem{rem}[thm]{\protect\remarkname}
  \theoremstyle{plain}
  \newtheorem{cor}[thm]{\protect\corollaryname}

\usepackage{upgreek}
\usepackage{amssymb}
\let\originalleft\left
\let\originalright\right
\renewcommand{\left}{\mathopen{}\mathclose\bgroup\originalleft}
\renewcommand{\right}{\aftergroup\egroup\originalright}
\makeatletter
\@namedef{subjclassname@2020}{%
  \textup{2020} Mathematics Subject Classification}
\makeatother

\makeatother

\usepackage{babel}
  \addto\captionsbritish{\renewcommand{\corollaryname}{Corollary}}
  \addto\captionsbritish{\renewcommand{\definitionname}{Definition}}
  \addto\captionsbritish{\renewcommand{\examplename}{Example}}
  \addto\captionsbritish{\renewcommand{\lemmaname}{Lemma}}
  \addto\captionsbritish{\renewcommand{\propositionname}{Proposition}}
  \addto\captionsbritish{\renewcommand{\remarkname}{Remark}}
  \addto\captionsbritish{\renewcommand{\theoremname}{Theorem}}
  \addto\captionsenglish{\renewcommand{\corollaryname}{Corollary}}
  \addto\captionsenglish{\renewcommand{\definitionname}{Definition}}
  \addto\captionsenglish{\renewcommand{\examplename}{Example}}
  \addto\captionsenglish{\renewcommand{\lemmaname}{Lemma}}
  \addto\captionsenglish{\renewcommand{\propositionname}{Proposition}}
  \addto\captionsenglish{\renewcommand{\remarkname}{Remark}}
  \addto\captionsenglish{\renewcommand{\theoremname}{Theorem}}
  \providecommand{\corollaryname}{Corollary}
  \providecommand{\definitionname}{Definition}
  \providecommand{\examplename}{Example}
  \providecommand{\lemmaname}{Lemma}
  \providecommand{\propositionname}{Proposition}
  \providecommand{\remarkname}{Remark}
\providecommand{\theoremname}{Theorem}

\begin{document}

\title[magnitudes, scalable monoids and quantity spaces]{magnitudes, scalable monoids \\
and quantity spaces}

\address{Dan Jonsson, University of Gothenburg, Gothenburg, Sweden.}

\email{dan.jonsson@gu.se}

\author{Dan Jonsson}
\begin{abstract}
In ancient Greek mathematics, magnitudes such as lengths were strictly
distinguished from numbers. In modern quantity calculus, a distinction
is made between quantities and scalars that serve as measures of quantities.
It can be argued that quantities should play a more prominent, independent
role in modern mathematics, as magnitudes earlier.

The introduction includes a sketch of the development and structure
of the pre-modern theory of magnitudes and numbers. Then, a scalable
monoid over a ring is defined and its basic properties are described.
Congruence relations on scalable monoids, direct and tensor products
of scalable monoids, subalgebras and homomorphic images of scalable
monoids, and unit elements of scalable monoids are also defined and
analyzed.

A quantity space is defined as a commutative scalable monoid over
a field, admitting a finite basis similar to a basis for a free abelian
group. The mathematical theory of quantity spaces forms the basis
of a rigorous quantity calculus and is developed with a view to applications
in metrology and foundations of physics.
\end{abstract}

\maketitle

\section{Introduction and historical background}

Equations such as $E=\tfrac{mv^{2}}{2}$ or $\frac{\partial T}{\partial t}=\kappa\frac{\partial^{2}T}{\partial x^{2}}$,
used to express physical laws, describe relationships between scalars,
commonly real numbers. An alternative interpretation is possible,
however. Since the scalars assigned to the variables in these equations
are numerical measures of certain quantities, the equations express
relationships between these quantities as well. For example, $E=\tfrac{mv^{2}}{2}$
can also be interpreted as describing a relation between an energy
$E$, a mass $m$ and a velocity $v$ \textendash{} three underlying
physical quantities, whose existence and properties do not depend
on scalars used to represent them. With this interpretation, though,
$\tfrac{mv^{2}}{2}$ and similar expressions will have meaning only
if operations on quantities, corresponding to operations on numbers,
are defined. In other words, an appropriate way of calculating with
quantities, a \emph{quantity calculus}, needs to be available.

In a useful survey \cite{BOER}, de Boer described the development
of quantity calculus until the late 20th century, starting with Maxwell's
\cite{MAXW} concept of a physical quantity $q$ comprised of a unit
quantity $\left[q\right]$ of the same kind as $q$ and a scalar $\left\{ q\right\} $
which is the measure of $q$ relative to $\left[q\right]$, so that
we can write $q=\left\{ q\right\} \left[q\right]$. Like Lodge \cite{LOD},
in a seminal article in \emph{Nature} 1888, and Wallot \cite{WAL},
in an influential article in \emph{Handbuch der Physik} 1926, de Boer
argued, however, that the physical quantity should be seen as a primitive
notion: \textcolor{black}{ontologically, the quantity precedes the
measure used to describe it, and the assignment $q=\left\{ q\right\} \left[q\right]$
can be used to specify a particular quantity but not to define the
notion of a quantity}%
\textcolor{black}{{} \cite[pp. 1--2]{JON1}.}

Actually, the roots of quantity calculus go far deeper in the history
of mathematics than to Wallot, Lodge, Maxwell or even other scientists
of the modern era, such as Fourier \cite{FOUR}; its origins can be
traced back to ancient Greek geometry and arithmetic, as codified
in Euclid's \emph{Elements \cite{EUCL}}. 

Of fundamental importance in the \emph{Elements} is the distinction
between \emph{numbers (multitudes)} and \emph{magnitudes}. The notion
of a number (\emph{arithmos}) is based on that of a ''unit'' or
''monad'' (\emph{monas}); a number is ''a multitude composed of
units''. Thus, a number is essentially a positive integer. (A collection
of units containing just one unit was not, in principle, considered
to be a multitude of units in Greek arithmetic, so strictly speaking
$1$ was not a number.) Numbers can be compared, added and multiplied,
and a smaller number can be subtracted from a larger one, but the
ratio of two numbers $m,n$ is not itself a number but just a pair
$m:n$ expressing relative size. (The ratio of integers is not necessarily
an integer.) Ratios can, however, be compared; $m:n=m':n'$ means
that $mn'=nm'$. A bigger number $m$ is said to be ''measured''
by a smaller number $k$ if $m=rk$ for some number $r$; a prime
number is a number that is not measured by any other number (or measured
only by $1$), and $m,n$ are relatively prime when there is no number
(except 1) measuring both. 

Magnitudes (\emph{megethos}), on the other hand, are phenomena such
as lengths, areas, volumes or times. Unlike numbers, magnitudes are
of different kinds\emph{, }and while the magnitudes of a particular
kind correspond loosely to numbers, making measurement of magnitudes
possible, the magnitudes form a continuum, and there are no distinguished
''unit magnitudes''. In Greek mathematics, magnitudes of the same
kind can be compared and added, and a smaller magnitude can be subtracted
from a larger one of the same kind, but magnitudes cannot, in general,
be multiplied or divided. One can form the ratio of two magnitudes
of the same kind, $p$ and $q$, but this is not a magnitude but just
a pair $p:q$ expressing relative size. A greater magnitude $q$ is
said to be measured by a smaller magnitude $u$ if there is a number
$n$ such that $q$ is equal to $u$ taken $n$ times; we may write
this as $q=n\times u$ here.

Remarkably, the first three propositions about magnitudes proved by
Euclid in the \emph{Elements} are, in the notation used here, 
\begin{gather*}
n\times(u_{1}\dotplus\cdots\dotplus u_{k})=n\times u_{1}\dotplus\cdots\dotplus n\times u_{k},\\
(n_{1}+\cdots+n_{k})\times u=n_{1}\times u\dotplus\cdots\dotplus n_{k}\times u,\qquad m\times(n\times u)=(mn)\times u,
\end{gather*}
where $m,n,n_{1},\ldots,n_{k}$ are numbers (\emph{arithmoi}), $u$
is a magnitude, $u_{1},\ldots,u_{k}$ are magnitudes of the same kind,
$\dotplus$ denotes the sum of magnitudes of the same kind, and $\times$
denotes the product of a number and a magnitude. As shown in Section
\ref{subsec:Commensurability-classes-with}, these identities are
fundamental in modern quantity calculus as well.

If $p$ and $q$ are magnitudes of the same kind, and there is some
magnitude $u$ of this kind and some numbers $m,n$ such that $p=m\times u$
and $q=n\times u$, then $p$ and $q$ are said to be ''commensurable''.
The ratio of magnitudes $p:q$ can then be represented by the ratio
of numbers $m:n$, assumed to be unique (unlike the two numbers specifying
the ratio). %
However, magnitudes may also be ''relatively prime''; it may happen
that $p:q$ cannot be expressed as $m:n$ for any numbers $m,n$ because
there are no $m,n,u$ such that $p=m\times u$ and $q=n\times u$.
In view of the Pythagorean philosophical conviction of the primacy
of numbers, the discovery of examples of such ''incommensurable''
magnitudes created a deep crisis in early Greek mathematics \cite{HASS},
a crisis that also affected the foundations of geometry. 

If ratios of \emph{arithmoi} do not always suffice to represent ratios
of magnitudes, it seems that it would not always be possible to express
in terms of \emph{arithmoi} the fact that two ratios of magnitudes
are equal, as are the ratios of the lengths of corresponding sides
of similar triangles. This difficulty was resolved by Eudoxos, who
realized that a ''proportion'', that is, a relation among magnitudes
of the form \emph{''}$\,p$~is to $q$ as $p'$ is to $q'$\emph{''},
conveniently denoted $p\vcentcolon q\dblcolon p'\vcentcolon q'$,
can be defined numerically even if there is no pair of ratios of \emph{arithmoi}
$m:n$ and $m':n'$ corresponding to $p:q$ and $p':q'$, respectively,
so that $p:q\dblcolon p':q'$ cannot be inferred from $m:n=m':n'$.
Specifically, as described in Book V of the \emph{Elements}, Eudoxos
invented an ingenious indirect way of determining if $p\vcentcolon q\dblcolon p'\vcentcolon q'$
in terms of nothing but \emph{arithmoi} by means of a construction
similar to the Dedekind cut \cite{HASS}. Using modern terminology,
one can say that Eudoxos defined an equivalence relation $\!\dblcolon$
between pairs of magnitudes of the same kind in terms of positive
integers, and as a consequence it became possible to conceptualize
in terms of \emph{arithmoi} not only ratios of magnitudes corresponding
to rational numbers but also ratios of magnitudes corresponding to
irrational numbers. Eudoxos thus reconciled the continuum of magnitudes
with the discrete \emph{arithmoi}, but in retrospect this feat reduced
the incentive to rethink the Greek notion of number, to generalize
the \emph{arithmoi}.

To summarize, Greek mathematicians used two notions of muchness, and
built a theoretical system around each notion. These systems were
connected by relationships of the form $q=n\times u$, where $q$
is a magnitude, $n$ a number and $u$ a magnitude of the same kind
as $q$, foreboding from the distant past Maxwell's quantity formula
$q=\left\{ q\right\} \left[q\right]$, although Euclid wisely did
not propose to define magnitudes in terms of units and numbers. %
{} 

\medskip{}

The modern theory of numbers dramatically extends the theory of numbers
in the \emph{Elements}. Many types of numbers other than positive
integers have been added, and the notion of a number as an element
of an algebraic system has come to the forefront. The modern notion
of number was not developed by a straight-forward extension of the
concept of \emph{arithmos}, however; the initial development of the
new notion of number during the Renaissance was strongly inspired
by the ancient theory of magnitudes.

The beginning of the Renaissance saw renewed interest in the classical
Greek theories of magnitudes and numbers as known from Euclid's \emph{Elements},
but later these two notions gradually fused into that of a real number.
Malet \cite{MAL} remarks:
\begin{quotation}
{\footnotesize{}As far as we know, not only was the neat and consistent
separation between the Euclidean notions of numbers and magnitudes
preserved in Latin medieval translations {[}...{]}, but these notions
were still regularly taught in the major schools of Western Europe
in the second half of the 15th century. By the second half of the
17th century, however, the distinction between the classical notions
of (natural) numbers and continuous geometrical magnitudes was largely
gone, as were the notions themselves. {[}pp. 64\textendash 65{]}}{\footnotesize \par}
\end{quotation}
The force driving this transformation was the need for a continuum
of numbers as a basis for computation; the discrete \emph{arithmoi}
were not sufficient. As magnitudes of the same kind form a continuum,
the idea emerged that numbers should be regarded as an aspect of magnitudes.
''Number is to magnitude as wetness is to water'' said Stevin in
\emph{L'Arithmétique} \cite{STEV}, published 1585, and defined a
number as ''that by which one can tell the quantity of anything''
(cela, par lequel s\textquoteright explique la quantité de chascune
chose) {[}Definition II{]}. Thus, numbers were seen to form a continuum
by virtue of their intimate association with magnitudes.

Stevin's definition of a number is rather vague, and it is difficult
to see how a magnitude can be associated with a definite number, considering
that the numerical measure of a magnitude depends on a choice of a
unit magnitude. The notion of number was, however, refined during
the 17th century. In \emph{La Geometrie} \cite{DESC}, where Descartes
laid the groundwork for analytic geometry, he implicitly identified
numbers with \emph{ratios} of two magnitudes, namely lengths of line
segments, one of which was considered to have unit length, and in
\emph{Universal Arithmetick} \cite{NEWT} Newton, who had studied
both Euclid and Descartes, defined a number as follows:
\begin{quote}
{\footnotesize{}By }\emph{\footnotesize{}Number}{\footnotesize{} we
mean, not so much a Multitude of Unities, as the abstracted }\emph{\footnotesize{}Ratio}{\footnotesize{}
of any Quantity, to another Quantity of the same Kind, which we take
for Unity. {[}p. 2{]}}{\footnotesize \par}
\end{quote}
By assigning the number 1 to a unit quantity, the representation of
quantities by numbers is normalized, addressing a problem with Stevin's
definition. Also, a ratio of quantities of the same kind is a ''dimensionless''
quantity. Systems of such quantities contain a canonical unit quantity
$\boldsymbol{1}$, and addition, subtraction, multiplication and division
of dimensionless quantities yield dimensionless quantities. Hence,
a number and the corresponding dimensionless quantity are quite similar,
though Newton hints at a difference by calling numbers ''abstracted''
ratios of quantities.

Magnitudes, or ''dimensionful'' quantities, were thus needed only
as a scaffolding for the new notion of numbers, and when this notion
had been established its origins fell into oblivion and magnitudes
fell out of fashion. The tradition from Euclid paled away, but the
idea that numbers specify quantities relative to other quantities
remained, as in \cite{EUL}. A new theory of quantities originated
from this idea.

\medskip{}

While the Greek theory of magnitudes derived from geometry, the new
theory of quantities found applications in mathematical physics, a
branch of science that emerged in the 18th century. In \emph{The Analytic
Theory of Heat} \cite{FOUR2}, published in 1822 as \emph{Théorie
analytique de la Chaleur}, Fourier explains how physical quantities
relate to the numbers in his equations:
\begin{quotation}
{\small{}In order to measure these quantities and express them numerically,
they must be compared with different kinds of units, five in number,
namely, the unit of length, the unit of time, that of temperature,
that of weight, and finally the unit which serves to measure quantities
of heat. {[}pp. 126\textendash 127{]}}{\small \par}
\end{quotation}
We recognize here the ideas that there are quantities of different
kinds and that the number associated with a quantity depends on the
choice of a unit quantity of the same kind. 

Using the modern notion of, for example, a real number, we can generalize
relationships of the form $q=n\times u$, where $n$ is an \emph{arithmos}
and $u$ is a magnitude that measures (divides) $q$, to relationships
of the form $q=\mu\cdot u$, where $u$ is a freely chosen unit quantity
of the same kind as $q$, and $\mu$ is the measure of $q$ relative
to $u$, a number specifying the size of $q$ compared to $u$. If
$q=\mu\cdot u$ then $\mu$ is determined by $q$ and $u$, and we
may write $\mu=f\left(q,u\right)$ as $\mu=q/u$.

Fourier realized that the measure of a quantity may be defined in
terms of measures of other quantities, in turn dependent on the units
for these quantities. For example, the measure of a velocity depends
on a unit of length $u_{\ell}$ and a unit of time $u_{t}$ since
a velocity is defined in terms of a length and a time, and the measure
of an area indirectly depends on a unit of length $u_{\ell}$. 

Formally, let the measure $\mu_{v}$ of a velocity $v$ relative to
$u_{v}$ be given by $\mu_{v}=F\left(\mu_{\ell},\mu_{t}\right)=F\left(\ell/u_{\ell},t/u_{t}\right)$,
where $F\left(x,y\right)=xy^{-1}$, and let the measure $\mu_{a}$
of the area $a$ of a rectangle relative to $u_{a}$ be given by $\mu_{a}=G\left(\mu_{\ell},\mu_{w}\right)=G\left(\ell/u_{\ell},w/u_{\ell}\right)$,
where $G\left(x,y\right)=xy$. Generalizing the magnitude identity
$m\times\left(n\times u\right)=mn\times u$, we have $M\cdot\left(N\cdot u\right)=MN\cdot u$
for any real numbers $M,N$. Thus, if $q=\mu\cdot u$ and $M>0$ then
$q=\left(M\mu M^{-1}\right)\cdot u=M\mu\cdot\left(M^{-1}\cdot u\right)$,
so $M\mu=q/\left(M^{-1}\cdot u\right)$, so it follows from the definitions
of $F$ and $G$ that, for any non-zero numbers $L,T$,
\begin{gather*}
\mu_{v}'=F\left(\ell/\left(L^{-1}\cdot u_{\ell}\right),t/\left(T^{-1}\cdot u_{t}\right)\right)=F\left(L\mu_{\ell},T\mu_{t}\right)=LT^{-1}F\left(\mu_{\ell},\mu_{t}\right)=LT^{-1}\mu_{v},\\
\mu_{a}'=G\left(\ell/\left(L^{-1}\cdot u_{\ell}\right),w/\left(L^{-1}\cdot u_{\ell}\right)\right)=G\left(L\mu_{\ell},L\mu_{w}\right)=L^{2}G\left(\mu_{\ell},\mu_{w}\right)=L^{2}\mu_{a}.
\end{gather*}

The two equations show how the measures $\mu_{v}$ and $\mu_{a}$
are affected by a change of units $u_{\ell}\mapsto L\cdot u_{\ell}$
and $u_{t}\mapsto T\cdot u_{t}$. Reasoning similarly \cite[pp. 128--130]{FOUR2},
Fourier pointed out that quantity terms can be equal or combined by
addition or subtraction only if they agree with respect to each \emph{exposant
de dimension}, having identical patterns of exponents in expressions
such as $LT^{-1}$, $LT^{-2}$ or $L^{2}$, since otherwise the validity
of numerical equations corresponding to quantity equations would depend
on an arbitrary choice of units. He thus introduced the principle
of dimensional homogeneity for equations that relate quantities.

Note that if $q=\mu\cdot u$ then $M\cdot q=M\cdot\left(\mu\cdot u\right)=M\mu\cdot u$,
so $\left(M\cdot q\right)/u=M\mu=M\left(q/u\right)$. Thus, in a sense
turning Fourier's argument around, we also have
\begin{gather*}
\mu_{v}'\!=\!F\left(\left(L\cdot\ell\right)/u_{\ell},\left(T\cdot t\right)/u_{t}\right)\!=\!F\left(L\left(\ell/u_{\ell}\right),T\left(t/u_{t}\right)\right)\!=\!LT^{-1}F\left(\ell/u_{\ell},t/u_{t}\right)\!=\!LT^{-1}\mu_{v},\\
\mu_{a}'\!=\!G\left(\left(L\cdot\ell\right)/u_{\ell},\left(L\cdot w\right)/u_{\ell}\right)\!=\!G\left(L\left(\ell/u_{\ell}\right),L\left(w/u_{\ell}\right)\right)\!=\!L^{2}G\left(\ell/u_{\ell},w/u_{\ell}\right)\!=\!L^{2}\mu_{a}.
\end{gather*}
These equations show how $\mu_{v}$ and $\mu_{a}$ are affected when
quantities change according to $\ell\mapsto L\cdot\ell$, $t\mapsto T\cdot t$.
For any fixed units $u_{\ell}$ and $u_{t}$, we can express this
as
\begin{gather*}
\varPhi\left(L\cdot\ell,T\cdot t\right)=LT^{-1}\varPhi\left(\ell,t\right),\\
\varGamma\left(L\cdot\ell,L\cdot w\right)=L^{2}\varGamma\left(\ell,w\right),
\end{gather*}
where $\varPhi$ and $\varGamma$ are the quantity-valued functions
given by $\varPhi\left(\ell,t\right)=F\left(\ell/u_{\ell},t/u_{t}\right)\cdot u_{v}$
and $\varGamma\left(\ell,w\right)=G\left(\ell/u_{\ell},w/u_{\ell}\right)\cdot u_{a}$,
respectively; note that $u_{v}$ and $u_{a}$ are also fixed since
they depend on $u_{\ell}$ and $u_{t}$.

The bilinearity properties of $\varPhi$ and $\varGamma$ suggest
that we write $\varPhi\left(\ell,t\right)$ as $\alpha\ell t^{-1}$
and $\varGamma\left(\ell,w\right)$ as $\beta\ell w$, where $\alpha$
and $\beta$ are numerical constants. Generalizing this heuristic
argument, we may introduce the idea that quantities of the same or
different kinds can be multiplied and divided, suggesting that we
can form arbitrary expressions of the form $\mu\prod_{i=1}^{n}q_{i}^{k_{i}},$
where $\mu$ is any number, $q_{i}$ are quantities and $k_{i}$ are
integers, thus coming close to the quantity calculus set out below.
Note, however, that Fourier did not actually define multiplication
or division of quantities as such. This came later, with Lodge \cite{LOD},
Wallot \cite{WAL} and others.

In retrospect, one may say that Fourier reinvented magnitudes as proto\-quantities
and extended the range of applications. While Fourier reasoned in
terms of multiplication and division of \emph{measures} of quantities,
he made a clear distinction between a quantity and its measure relative
to a unit, this measure being a real number rather than an \emph{arithmos},
he distinguished different kinds of quantities, and he considered
new kinds of quantities such as temperatures and amounts of heat.
Essential elements of a modern quantity calculus treating general
quantities as mathematical objects (almost) as real as numbers were
thus recognized early in the 19th century. 

Subsequent progress in this area of mathematics has not been fast
and straight-forward, however. A Euclidean synthesis did not emerge;
in his survey from 1994 de Boer concluded that ''a satisfactory axiomatic
foundation for the quantity calculus'' had not yet been formulated
\cite{BOER}.\medskip{}

Gowers \cite{GOV1} points out that many mathematical objects are
not defined directly by describing their essential properties, but
indirectly by \emph{construction-definitions}, specifying constructions
that can be shown to have these properties. For example, an ordered
pair $\left(x,y\right)$ may be defined by a construction-definition
as a set $\left\{ x,\left\{ y\right\} \right\} $; it can be shown
that this construction has the required properties, namely that $\left(x,y\right)=\left(x',y'\right)$
if and only if $x=x'$ and $y=y'$. Many contemporary formalizations
of the notion of a quantity (e.g., \cite{CARL,LAN}) use definitions
relying on constructions, often defining quantities in terms of scalar-unit
pairs%
{} in the tradition from Maxwell. (See also the survey in Appendix B.)
However, this is rather like defining a vector as a coordinates-basis
pair rather than as an element of a vector space, the modern definition.

Although magnitudes are illustrated by line segments in the \emph{Elements},
the notion of a magnitude is abstract and general. Remarkably, Euclid,
following Eudoxos, dealt with this notion in a very modern way. While
Euclid carefully defined other important objects such as points, lines
and numbers in terms of inherent properties, there is no statement
about what a magnitude \emph{''$\,$is''}. Instead, magnitudes are
characterized by how they relate to other magnitudes through their
roles in a system of magnitudes, to paraphrase Gowers \cite{GOW2}. 

In the same spirit, that of modern algebra, quantities are defined
in this article simply as elements of a ''quantity space''. Thus,
the focus is moved from individual quantities and operations on them
to the systems to which the quantities belong, meaning that the notion
of quantity calculus will give way to that of a quantity space. This
article considers the notion of a quantity space introduced in \cite{JON1}
and developed further in \cite{JON3}.

In the conceptual framework of universal algebra, a quantity space
is just a special \emph{scalable monoid} $\left(X,\mathsf{\ast},(\omega{}_{\lambda})_{\lambda\in R},1_{X}\right)$,
where $X$ is the underlying set of the algebra, $\left(X,\mathsf{*},1_{X}\right)$
is a monoid, $R$ is a fixed ring and every $\omega_{\lambda}$ is
unary operation on $X$. Writing $\ast\left(x,y\right)$ as $xy$
and denoting $\omega_{\lambda}\left(x\right)$ by $\lambda\cdot x$,
we have $1\cdot x=x$, $\lambda\cdot\left(\kappa\cdot x\right)=\lambda\kappa\cdot x$
and $\lambda\cdot xy=\left(\lambda\cdot x\right)y=x\left(\lambda\cdot y\right)$
for all $\lambda\in R$, $x,y\in X$. 

The relation $\sim$ on a scalable monoid $X$ defined by $x\sim y$
if and only if $\alpha\cdot x=\beta\cdot y$ for some $\alpha,\beta\in R$
is a congruence on $X$, so $X$ is partitioned into corresponding
\textcolor{black}{equivalence classes.} %
There is no global operation $\left(x,y\right)\mapsto x+y$ defined
on $X$, but within each equivalence class that contains a ''unit
element'' addition of its elements is induced by the addition in
$R$ (see Section \ref{subsec:Commensurability-classes-with}), and
multiplication of equivalence classes is induced by the multiplication
of elements of $X$ (see Section \ref{subsec:13}).

\emph{Quantity spaces} are to scalable monoids as vector spaces are
to modules. Specifically, a quantity space $Q$ is a\textcolor{cyan}{\emph{
}}\textcolor{black}{commutative} scalable monoid over a field, such
that there exists a finite basis for $Q$, similar to a basis for
a free abelian group. As noted, quantities are just elements of quantity
spaces, and dimensions are equivalence classes in quantity spaces.

\textcolor{black}{The remainder of this article is divided into two
main sections, namely Section 2 which deals with scalable monoids
and Section 3 where scalable monoids are specialized to quantity spaces.
There are also two Appendices, one of which relates the theory presented
here to contemporary research on quantity calculus.\pagebreak{}}

\section{\label{sec:2}Scalable monoids}

A scalable monoid is a monoid whose elements can be multiplied by
elements of a ring, and where multiplication in the monoid, multiplication
in the ring, and multiplication of monoid elements by ring elements
are compatible operations.

Scalable monoids are formally defined and compared to rings and modules
in Section \ref{s11}, and some basic facts about them are presented
in Section \ref{s12}. Sections \ref{subsec:13} and \ref{subsec:Quotients-of-scalable}
are concerned with congruences on scalable monoids and related notions
such as commensurability, orbitoids,%
{} homomorphisms and quotient algebras, while direct and tensor products
of scalable monoids are defined in Section \ref{s16}. Scalable monoids
with unit elements are investigated in Sections \ref{subsec:Commensurability-classes-with}
and \ref{s25}. In particular, addition of elements in the same equivalence\textcolor{brown}{{}
}%
class is defined, and coherent systems of unit elements are discussed.

\subsection{\label{s11}Mathematical background, main definition and simple examples}

A unital associative algebra $X$ over a (unital, associative but
not necessarily commutative) ring $R$ can be defined as a set, also
denoted $X$, with three operations: 
\begin{enumerate}
\item \emph{addition} of elements of $X$, a binary operation $+:\left(x,y\right)\mapsto x+y$
on $X$ such that $X$ equipped with $+$ is an abelian group; 
\item \emph{multiplication} of elements of $X$, a binary operation $\ast:\left(x,y\right)\mapsto xy$
on $X$ such that $X$ equipped with $\ast$ is a monoid; 
\item \emph{scalar multiplication} of elements of $X$ by elements of $R,$
a monoid action $\left(\alpha,x\right)\mapsto\alpha\cdot x$ where
the multiplicative monoid of $R$ acts on $X$ so that $1\cdot x=x$
and $\alpha\cdot\left(\beta\cdot x\right)=\alpha\beta\cdot x$ for
all $\alpha,\beta\in R$ and $x\in X$. 
\end{enumerate}
There are identities specifying a link between each pair of operations: 
\begin{enumerate}
\item[(a)] addition and multiplication of elements of $X$ are linked by the
distributive laws $x\left(y+z\right)=xy+xz$ and $\left(x+y\right)z=xz+yz$; 
\item[(b)] addition of elements of $X$ or $R$ and scalar multiplication of
elements of $X$ by elements of $R$ are linked by the distributive
laws $\alpha\cdot\left(x+y\right)=\alpha\cdot x+\alpha\cdot y$ and
$\left(\alpha+\beta\right)\cdot x=\alpha\cdot x+\beta\cdot x$; 
\item[(c)] multiplication of elements of $X$ and scalar multiplication of elements
of $X$ by elements of $R$ are linked by the bilinearity laws $\alpha\cdot xy=\left(\alpha\cdot x\right)y$
and $\alpha\cdot xy=x\left(\alpha\cdot y\right)$.
\end{enumerate}
Related algebraic structures can be obtained from unital associative
algebras by removing one of the operations (1) \textendash{} (3) and
hence the links between the removed operation and the two others.
Two cases are very familiar: a\emph{ }ring has only addition and multiplication
of elements of $X$, linked as described in (a), and a\emph{ }(left)
module has only addition of elements of $X$ and scalar multiplication
of elements of $X$ by elements of $R$, linked as described in (b).
The question arises whether it would be meaningful and useful to define
an ``algebra without an additive group'', with only multiplication
of elements of $X$ and scalar multiplication of elements of $X$
by elements of $R$, linked as described in (c). 

The answer is affirmative. It turns out that this notion, a ''scalable
monoid'', formally related to rings and in particular modules, makes
sense mathematically and is remarkably well suited for modeling systems
of quantities. The ancient \emph{arithmos-megethos} pair of notions
receives a modern interpretation: while numbers can be formalized
as elements of rings, typically fields, quantities can be formalized
as elements of scalable monoids, specifically quantity spaces.
\begin{defn}
\label{thm:def1}Let $R$ be a (unital, associative) ring. A \emph{scalable
monoid} \emph{over} $R$ is a monoid $X$ equipped with a \emph{scaling
action} 
\[
\omega:R\times X\rightarrow X,\qquad\left(\alpha,x\right)\mapsto\alpha\cdot x,
\]
such that $1\cdot x=x$, $\alpha\cdot\left(\beta\cdot x\right)=\alpha\beta\cdot x$
and $\alpha\cdot xy=\left(\alpha\cdot x\right)y=x\left(\alpha\cdot y\right)$. 
\end{defn}

We denote the identity element of $X$ by $1_{\!X}$, and set $x^{0}=1_{\!X}$
for any $x\in X$. An \emph{invertible} element of a scalable monoid
$X$ is an element $x\in X$ that has a (necessarily unique) \emph{inverse}
$x^{-1}\in X$ such that $xx^{-1}=x^{-1}x=1_{\!X}$.

It is easy to verify that the \emph{trivial scaling action} of a ring
$R$ on a monoid $X$ defined by $\lambda\cdot x=x$ for all $\lambda\in R$
and $x\in X$ is indeed a scaling action according to Definition \ref{thm:def1}.
We call a monoid equipped with a trivial scaling action a \emph{trivially
scalable} monoid. A scalable monoid of this kind is essentially just
a monoid, since the operation $\left(\lambda,x\right)\mapsto\lambda\cdot x$
can be disregarded in this case. 
\begin{example}
A \emph{trivial} scalable monoid is a trivial monoid $\left\{ 1_{\!X}\right\} $
with a trivial scaling action.
\end{example}

\begin{example}
\label{ex1-1}Let $M\left(n\right)$ be the multiplicative monoid
of all $n\times n$ matrices with entries in $\mathbb{R}$. Then $M\left(n\right)$
is a scalable monoid over the corresponding matrix ring $R\left(n\right)$,
with the scaling action defined by $\mathbf{A}\cdot\mathbf{X}=\left(\det\mathbf{A}\right)\mathbf{X}$.
\end{example}

\begin{example}
\label{ex1}Let $R\left\llbracket x_{1};\ldots;x_{n}\right\rrbracket $
denote the set of all monomials of the form
\[
\lambda x_{1}^{k_{1}}\ldots x_{n}^{k_{n}},
\]
where $R$ is a commutative ring, $\lambda\in R$, $x_{1},\ldots,x_{n}$
are uninterpreted symbols and $k_{1},\ldots,k_{n}$ are non-negative
integers. We can define the operations $\left(s,t\right)\mapsto st$,
$\left(\alpha,t\right)\mapsto\alpha\cdot t$ and $\left(\right)\mapsto1_{R\left\llbracket x_{1};\ldots;x_{n}\right\rrbracket }$
on $R\left\llbracket x_{1};\ldots;x_{n}\right\rrbracket $ by setting
\begin{gather*}
\left(\lambda x_{1}^{j_{1}}\ldots x_{n}^{j_{n}}\right)\left(\kappa x_{1}^{k_{1}}\ldots x_{n}^{k_{n}}\right)=\left(\lambda\kappa\right)x_{1}^{\left(j_{1}+k_{1}\right)}\ldots x_{n}^{\left(j_{n}+k_{n}\right)},\\
\alpha\cdot\lambda x_{1}^{k_{1}}\ldots x_{n}^{k_{n}}=\left(\alpha\lambda\right)x_{1}^{k_{1}}\ldots x_{n}^{k_{n}},\\
1_{R\left\llbracket x_{1};\ldots;x_{n}\right\rrbracket }=1x_{1}^{0}\ldots x_{n}^{0}.
\end{gather*}
$R\left\llbracket x_{1};\ldots;x_{n}\right\rrbracket $ equipped with
these operations is a commutative scalable monoid over $R$.
\end{example}

\subsection{\label{s12}Some basic facts about scalable monoids}

A not necessarily commutative scalable monoid over a not necessarily
commutative ring nevertheless exhibits certain commutativity properties
as described in the following useful lemma:
\begin{lem}
\label{thm:lem1}Let $X$ be a scalable monoid over $R$. For any
$x,y\in X$ and $\alpha,\beta\in R$ we have 
\[
\left(\alpha\cdot x\right)\left(\beta\cdot y\right)=\alpha\beta\cdot xy,\qquad\alpha\beta\cdot x=\alpha\cdot\left(\beta\cdot x\right)=\beta\cdot\left(\alpha\cdot x\right)=\beta\alpha\cdot x.
\]
\end{lem}

\begin{proof}
By Definition \ref{thm:def1}, 
\[
\left(\alpha\cdot x\right)\left(\beta\cdot y\right)=\alpha\cdot x\left(\beta\cdot y\right)=\alpha\cdot\left(\beta\cdot xy\right)=\alpha\beta\cdot xy,
\]
\begin{align*}
\alpha\beta\cdot x & =\alpha\cdot\left(\beta\cdot x\right)=\alpha\cdot\left(\beta\cdot1_{\!X}x\right)=\alpha\cdot\left(\beta\cdot1_{\!X}\right)x=\left(\beta\cdot1_{\!X}\right)\left(\alpha\cdot x\right)\\
 & =\beta\alpha\cdot1_{\!X}x=\beta\alpha\cdot x=\beta\cdot\left(\alpha\cdot x\right),
\end{align*}
 where the first identity is used in the proof of the second.
\end{proof}
For example, $\left(\left(\det\mathbf{A}\right)\mathbf{X}\right)\left(\left(\det\mathbf{B}\right)\mathbf{Y}\right)=\left(\det\mathbf{A}\mathbf{B}\right)\mathbf{XY}$
is obviously the identity $\left(\mathbf{A}\cdot\mathbf{X}\right)\left(\mathbf{B}\cdot\mathbf{Y}\right)=\mathbf{AB}\cdot\mathbf{XY}$
for the scalable monoid in Example \ref{ex1-1}.

Since every monoid $\mathfrak{M}$ has a unique identity element $1_{\mathfrak{M}}$,
the class of all monoids forms a variety of algebras with a binary
operation $\ast:\left(x,y\right)\mapsto xy$, a nullary operation
$1_{\mathfrak{M}}:\left(\right)\mapsto1_{\mathfrak{M}}$ and identities
\[
x\left(yz\right)=\left(xy\right)z,\quad1_{\mathfrak{M}}x=x=x1_{\mathfrak{M}}.
\]
The class of all scalable monoids over a fixed ring $R$ is a variety
in addition equipped with a set of unary operations $\left\{ \omega_{\lambda}\mid\lambda\in R\right\} $,
derived from the scaling action $\omega$ in Definition \ref{thm:def1}
by setting $\omega_{\lambda}\left(x\right)=\lambda\cdot x$ for all
$\lambda\in R$ and $x\in X$, and with the additional identities
\[
\omega_{1}\left(x\right)=x,\quad\omega_{\lambda}\left(\omega_{\kappa}\left(x\right)\right)=\omega_{\lambda\kappa}\left(x\right),\quad\omega_{\lambda}\left(xy\right)=\omega_{\lambda}\left(x\right)\,y=x\,\omega_{\lambda}\left(y\right)\qquad\left(\lambda,\kappa\in R\right),
\]
corresponding to $1\cdot x=x$, $\alpha\cdot\left(\beta\cdot x\right)=\alpha\beta\cdot x$
and $\alpha\cdot xy=\left(\alpha\cdot x\right)y=x\left(\alpha\cdot y\right)$.

The scalable monoids is thus a variety of algebras belonging to the
class of all \emph{$\left(R,\boldsymbol{1}\right)$-magmas}%
{} 
\[
\left(X,\ast,(\omega{}_{\lambda})_{\lambda\in R},1_{\!X}\right).
\]
where $X$ is a carrier set, $\ast$ a binary operation, $\omega_{\lambda}$
a unary operation and $1_{X}$ a nullary operation. %
The general definitions of subalgebras, homomorphisms and products
of algebras in the theory of universal algebras apply to \emph{$\left(R,\boldsymbol{1}\right)$-}magmas.
In particular, a subalgebra of an \emph{$\left(R,\boldsymbol{1}\right)$-}magma
$X$ is a subset $Y$ of $X$ such that $1_{\!X}\in Y$ and $xy,\lambda\cdot x\in Y$
for any $x,y\in Y$ and $\lambda\in R$. Also, a homomorphism $\phi:X\rightarrow Y$
of \emph{$\left(R,\boldsymbol{1}\right)$}-magmas $X$ and $Y$ is
a function such that $\phi\left(1_{\!X}\right)=1_{\!Y}$ and we have
$\phi\left(xy\right)=\phi\left(x\right)\phi\left(y\right)$ and $\phi\left(\lambda\cdot x\right)=\lambda\cdot\phi\left(x\right)$
for any $x,y\in X$ and $\lambda\in R$.

Recall, furthermore, that varieties are closed under the operations
of forming subalgebras, homomorphic images and products since the
defining identities are replicated by these operations \cite{BIRK}.
Thus, a subalgebra of a scalable monoid over $R$, a homomorphic image
of a scalable monoid over $R$, and a direct product of scalable monoids
over $R$ are all scalable monoids over $R$. Results related to these
and other constructions will be considered in the remainder of Section
\ref{sec:2}.

\subsection{\label{subsec:13}Commensurability classes and canonical quotients}

In ancient Greek mathematics, the notions of a sum, difference or
ratio of magnitudes did not apply to magnitudes of different kinds,
so in particular these could not be commensurable in the Greek (Pythagorean)
sense. Moreover, magnitudes of the same kind, for example, two lengths,
could nevertheless be incommensurable. In this section, we introduce
a seemingly more radical idea: quantities are of the same kind \emph{if
and only if} they are commensurable.
\begin{defn}
\label{d3.1}Given a scalable monoid $X$ over $R$, let $\sim$ be
the relation on $X$ such that $x\sim y$ if and only if $\alpha\cdot x=\beta\cdot y$
for some $\alpha,\beta\in R.$ We say that $x$ and $y$ are \emph{commensurable}
if and only if $x\sim y$; otherwise $x$ and $y$ are \emph{incommensurable}.
\end{defn}

Let $R\cdot x$ denote the set $\left\{ \lambda\cdot x\mid\lambda\in R\right\} $,
that is, the orbit of $x\in X$ for the scaling action $\omega:R\times X\rightarrow X$,
and let $\approx$ denote the relation on $X$ such that $x\approx y$
if and only if there is some $t\in X$ such that $x,y\in R\cdot t$.
Note that $\approx$ is not an equivalence relation; it is reflexive
since $x\in1\cdot x$ for all $x\in X$ and symmetric by construction
but not transitive, meaning that the orbits for $\omega$ may overlap.
On the other hand, $x\sim y$ if and only if $\left(R\cdot x\right)\cap\left(R\cdot y\right)\neq\emptyset$,
and this relation is indeed transitive. 
\begin{prop}
\label{s3.1} The relation $\sim$ on a scalable monoid $X$ over
$R$ is an equivalence relation.
\end{prop}

\begin{proof}
The relation $\sim$ is reflexive since $1\cdot x=1\cdot x$ for all
$x\in X$, symmetric by construction, and transitive because if $\alpha\cdot x=\beta\cdot y$
and $\gamma\cdot y=\delta\cdot z$ for some $x,y,z\in X$ and $\alpha,\beta,\gamma,\delta\in R$
then it follows from Lemma \ref{thm:lem1} that
\[
\gamma\alpha\cdot x=\gamma\cdot\left(\alpha\cdot x\right)=\gamma\cdot\left(\beta\cdot y\right)=\beta\cdot\left(\gamma\cdot y\right)=\beta\cdot\left(\delta\cdot z\right)=\beta\delta\cdot z,
\]
where $\gamma\alpha,\beta\delta\in R$. 
\end{proof}
\begin{defn}
A\emph{ commensurability class} or \emph{orbitoid} $\mathsf{C}$ is
an equivalence class for $\sim$. The orbitoid that contains $x$
is denoted by $\left[x\right]$, and $X/{\sim}$ denotes the set $\left\{ \left[x\right]\mid x\in X\right\} $.
\end{defn}

For example, the commensurability classes of $R\left\llbracket x_{1};\ldots;x_{n}\right\rrbracket $
in Example \ref{ex1} are the sets $\left\{ \lambda x_{1}^{k_{1}}\ldots x_{n}^{k_{n}}\mid\lambda\in R\right\} $
for fixed non-negative integers $k_{1},\ldots,k_{n}$. 
\begin{rem}
The orbits corresponding to an action of a group $G$ on a set $X$
are precisely the equivalence classes given by the equivalence relation
$\sim_{G}$ defined by $\sim_{G}$ if and only if $\alpha\cdot x=y$
for some $\alpha\in G$; we clearly have $G\cdot x=G\cdot y$ if and
only if $x\sim_{G}y$. Similarly, orbitoids \textendash{} generalized
orbits in $X$ under a monoid action satisfying $\alpha\cdot\left(\beta\cdot x\right)=\beta\cdot\left(\alpha\cdot x\right)$
\textendash{} are given by the equivalence relation $\sim$ defined
by $x\sim y$ if and only if $\alpha\cdot x=\beta\cdot y$ for some
$\alpha,\beta\in R$. One may say that orbitoids generalize orbits
as $\sim$ generalizes $\sim_{G}$.%
\end{rem}

\begin{prop}
\label{p22-2}If $x\sim y$ then $\lambda\cdot x\sim y$, $x\sim\lambda\cdot y$
and $\lambda\cdot x\sim\lambda\cdot y$ for all $\lambda\in R$.
\end{prop}

\begin{proof}
If $x\sim y$ then $\alpha\cdot x=\beta\cdot y$ for some $\alpha,\beta\in R$,
so by Lemma \ref{thm:lem1}
\[
\alpha\lambda\cdot x=\alpha\cdot\left(\lambda\cdot x\right)=\lambda\cdot\left(\alpha\cdot x\right)=\lambda\cdot\left(\beta\cdot y\right)=\beta\cdot\left(\lambda\cdot y\right)=\beta\lambda\cdot y,
\]
where $\alpha\lambda,\beta\lambda\in R$. 
\end{proof}
\begin{cor}
\label{c31}$\lambda\cdot x\sim x$ for all $x\in X$ and $\lambda\in R$.
\end{cor}

It is instructive to compare the present notion of commensurability
with the classical one. If $x=\alpha\cdot t$ and $y=\beta\cdot t$
then $\beta\cdot x=\beta\cdot\left(\alpha\cdot t\right)=\alpha\cdot\left(\beta\cdot t\right)=\alpha\cdot y$,
so if $x\approx y$ then $x\sim y$. We say that $x$ and $y$ are
\emph{strongly commensurable} if and only if $x\approx y$; otherwise,
$x$ and $y$ are said to be \emph{weakly incommensurable}. 

Incommensurability of magnitudes in the Pythagorean sense obviously
corresponds to weak incommensurability, so it is implied by, but does
not imply, incommensurability in the present sense. Conversely, we
have weakened the classical notion of commensurability here, at the
same time making commensurability into an equivalence relation. The
present concept of commensurability corresponds to the intuitive notion
of magnitudes of the same kind, or the somewhat fuzzy notion of quantities
of the same kind in modern theoretical metrology \cite{VIM}.
\begin{rem}
The mathematical quantities defined here are size-properties of certain
objects or phenomena. Through one or more abstraction steps, concrete
properties can be reduced to more abstract properties. The level of
abstraction chosen affects the categorization of quantities into kinds
of quantities. For example, it would seem that there are no scalars
$\alpha,\beta$ such that $\alpha\cdot x=\beta\cdot y$, where $x$
is a planar angle and $y$ a solid angle. A plane angle cannot be
resized to a solid angle, or vice versa, so plane and solid angles
would appear to be quantities of different kinds. However, this is
a conclusion based on concrete properties of plane and solid angles.
It is also possible to characterize them by commensurable abstract
size-properties, so that they become quantities of the same kind.
This is why both plane and solid angles are mainly regarded as ''dimensionless''
quantities (see also \cite[pp. 8-12]{JON3}).

\end{rem}

So far, we have regarded $\sim$ as an equivalence relation, but it
turns out that more can be said.
\begin{prop}
\label{s3.2}Let $X$ be a scalable monoid over $R$. The relation
$\sim$ is a congruence on $X$ with regard to the operations $\left(x,y\right)\mapsto xy$
and $\left(\lambda,x\right)\mapsto\lambda\cdot x$.
\end{prop}

\begin{proof}
If $\alpha\cdot x=\alpha'\cdot x'$ and $\beta\cdot y=\beta'\cdot y'$
for some $x,x',y,y'\in X$ and $\alpha,\alpha',\beta,\beta'\in R$
then $\left(\alpha\cdot x\right)\left(\beta\cdot y\right)=\left(\alpha'\cdot x'\right)\left(\beta'\cdot y'\right)$,
so $\alpha\beta\cdot xy=\alpha'\beta'\cdot x'y'$ by Lemma \ref{thm:lem1}.
As $\alpha\beta,\alpha'\beta'\in R$, this means that if $x\sim x'$
and $y\sim y'$ then $xy\sim x'y'$. Also, if $x\sim x'$ then $\lambda\cdot x\sim\lambda\cdot x'$
for any $\lambda\in R$ by Proposition \ref{p22-2}.
\end{proof}
In view of Proposition \ref{s3.2}, we can define operations on $X/{\sim}$
as follows:
\begin{defn}
\label{def:214}Set $\left[x\right]\left[y\right]=\left[xy\right]$,
$\lambda\cdot\left[x\right]=\left[\lambda\cdot x\right]$ and $1_{X/{\sim}}=\left[1_{\!X}\right]$
for any $\left[x\right],\left[y\right]\in X/{\sim}$ and $\lambda\in R$.
\end{defn}

{} Given these definitions, $X/{\sim}$ is an $\left(R,\boldsymbol{1}\right)$-magma
and the surjective function $\phi:X\rightarrow X/{\sim}$ given by
$\phi\left(x\right)=\left[x\right]$ satisfies the conditions 
\[
\phi\left(xy\right)=\phi\left(x\right)\phi\left(y\right),\quad\phi\left(\lambda\cdot x\right)=\lambda\cdot\phi\left(x\right),\quad\phi\left(1_{\!X}\right)=1_{X/{\sim}},
\]
so $\phi$ is a homomorphism of $\left(R,\boldsymbol{1}\right)$-magmas
and thus of scalable monoids.%
\begin{prop}
\label{s3.3-1}If $X$ is a scalable monoid over $R$ then $X/{\sim}$
is a scalable monoid over $R$, and the function
\[
\phi:X\rightarrow X/{\sim},\qquad x\mapsto\left[x\right],
\]
is a surjective homomorphism of scalable monoids.
\end{prop}

We call $X/{\sim}$ the \emph{canonical quotient} of $X$.
\begin{prop}
\label{p5-1}If $X$ is a scalable monoid then $X/{\sim}$ is a trivially
scalable monoid.
\end{prop}

\begin{proof}
By Corollary \ref{c31}, $\lambda\cdot\left[x\right]=\left[\lambda\cdot x\right]=\left[x\right]$
for all $\lambda\in R,x\in X$.
\end{proof}
In many situations, it is natural to regard $X/{\sim}$ as a monoid
with operations inherited from $X$ by setting $\left[x\right]\left[y\right]=\left[xy\right]$
and $1_{X/{\sim}}=\left[1_{\!X}\right]$%
.

\subsection{\label{s16}Direct and tensor products of scalable monoids}

Consider an $\left(R,\boldsymbol{1}\right)$-magma 
\[
\left(X\times Y,\ast,\left(\omega_{\lambda}\right)_{\lambda\in R},1_{X\times Y}\right)
\]
where $X$ and $Y$ denote the underlying sets of two scalable monoids
$X$ and $Y$ over $R$, $\ast$ is a binary operation given by $\left(x_{1},y_{1}\right)\left(x_{2},y_{2}\right)=\left(x_{1}x_{2},y_{1}y_{2}\right)$,
where $x_{1},x_{2}\in X$ and $y_{1},y_{2}\in Y$, each $\omega_{\lambda}$
is a unary operation given by $\omega_{\lambda}\left(x,y\right)=\lambda\cdot\left(x,y\right)=\left(\lambda\cdot x,\lambda\cdot y\right)$,
where $x\in X$ and $y\in Y$, and $1_{X\times Y}$ is a nullary operation
given by $1_{X\times Y}=\left(1_{\!X},1_{\!Y}\right)$. Straight-forward
calculations (or the HSP theorem \cite{BIRK}) show that this $\left(R,\boldsymbol{1}\right)$-magma,
likewise denoted $X\times Y$, is a scalable monoid over $R$. We
call $X\times Y$ the \emph{direct product} of $X$ and $Y$.

The direct product of scalable monoids is a generic product, applicable
to any universal algebra. Another kind of product, which exploits
the fact that $\left(\lambda\cdot x\right)y=\lambda\cdot xy=x\left(\lambda\cdot y\right)$
in scalable monoids, namely the \emph{tensor product,} turns out to
be more useful in many cases.
\begin{defn}
Given scalable monoids $X$ and $Y$ over $R$, let $\backsim_{\otimes}$
be the binary relation on $X\times Y$ such that $\left(x_{1},y_{1}\right)\backsim_{\otimes}\left(x_{2},y_{2}\right)$
if and only if $\left(\alpha\cdot x_{1},\beta\cdot y_{1}\right)=\left(\beta\cdot x_{2},\alpha\cdot y_{2}\right)$
for some $\alpha,\beta\in R$.
\end{defn}

\begin{prop}
Let $X$ and $Y$ be scalable monoids over $R$. Then $\backsim_{\otimes}$
is an equivalence relation on $X\times Y$. 
\end{prop}

\begin{proof}
$\backsim_{\otimes}$ is reflexive since $\left(1\cdot x,1\cdot y\right)=\left(1\cdot x,1\cdot y\right)$,
and symmetric by construction. If $\left(\alpha\cdot x_{1},\beta\cdot y_{1}\right)=\left(\beta\cdot x_{2},\alpha\cdot y_{2}\right)$
and $\left(\gamma\cdot x_{2},\delta\cdot y_{2}\right)=\left(\delta\cdot x_{3},\gamma\cdot y_{3}\right)$
then 
\begin{gather*}
\left(\gamma\cdot\left(\alpha\cdot x_{1}\right),\delta\cdot\left(\beta\cdot y_{1}\right)\right)=\left(\gamma\cdot\left(\beta\cdot x_{2}\right),\delta\cdot\left(\alpha\cdot y_{2}\right)\right),\\
\left(\beta\cdot\left(\gamma\cdot x_{2}\right),\alpha\cdot\left(\delta\cdot y_{2}\right)\right)=\left(\beta\cdot\left(\delta\cdot x_{3}\right),\alpha\cdot\left(\gamma\cdot y_{3}\right)\right).
\end{gather*}
By Lemma \ref{thm:lem1}, $\left(\gamma\cdot\left(\beta\cdot x_{2}\right),\delta\cdot\left(\alpha\cdot y_{2}\right)\right)=\left(\beta\cdot\left(\gamma\cdot x_{2}\right),\alpha\cdot\left(\delta\cdot y_{2}\right)\right)$,
and thus
\begin{gather*}
\left(\gamma\alpha\cdot x_{1},\delta\beta\cdot y_{1}\right)=\left(\beta\delta\cdot x_{3},\alpha\gamma\cdot y_{3}\right)=\left(\delta\beta\cdot x_{3},\gamma\alpha\cdot y_{3}\right),
\end{gather*}
where $\gamma\alpha,\delta\beta\in R$, so $\backsim_{\otimes}$ is
transitive as well.
\end{proof}
\begin{defn}
Let $X$ and $Y$ be scalable monoids, let $x\otimes y$ denote the
equivalence class 
\[
\left\{ \left(s,t\right)\mid\left(s,t\right)\backsim_{\otimes}\left(x,y\right)\right\} ,
\]
where $x\in X,y\in Y$, and let $X\otimes Y$ denote the set 
\[
\left\{ x\otimes y\mid x\in X,y\in Y\right\} .
\]
equipped with the operations given by
\[
\left(x_{1}\otimes y_{1}\right)\left(x_{2}\otimes y_{2}\right)=x_{1}x_{2}\otimes y_{1}y_{2},\quad\lambda\cdot x\otimes y=\left(\lambda\cdot x\right)\otimes y,\quad1_{X\otimes Y}=1_{\!X}\otimes1_{\!Y}.
\]
We call $X\otimes Y$ the \emph{tensor product }of\emph{ $X$ and
$Y$}.
\end{defn}

\begin{prop}
Let $X$ and $Y$ be scalable monoids over $R$, $x\in X$ and $y\in X$.
Then $\left(\lambda\cdot x\right)\otimes y=x\otimes\left(\lambda\cdot y\right)$
for every $\lambda\in R$.
\end{prop}

\begin{proof}
We have $\left(1\cdot\left(\lambda\cdot x\right),\lambda\cdot y\right)=\left(\lambda\cdot x,1\cdot\left(\lambda\cdot y\right)\right)$,
so $\left(\lambda\cdot x,y\right)\backsim_{\otimes}\left(x,\lambda\cdot y\right)$.
\end{proof}
\begin{prop}
Let $X$ and $Y$ be scalable monoids over $R$. Then $X\otimes Y$
is a scalable monoid over $R$.
\end{prop}

\begin{proof}
$X\otimes Y$ is a monoid since 
\begin{align*}
\left(1_{\!X}\otimes1_{\!Y}\right)\left(x\otimes y\right) & =1_{\!X}x\otimes1_{\!Y}y=x\otimes y=x1_{\!X}\otimes y1_{\!Y}=\left(x\otimes y\right)\left(1_{\!X}\otimes1_{\!Y}\right),
\end{align*}
\begin{align*}
\left(\left(x_{1}\otimes y_{1}\right)\left(x_{2}\otimes y_{2}\right)\right)\left(x_{3}\otimes y_{3}\right)\\
 & =\left(x_{1}x_{2}\otimes y_{1}y_{2}\right)\left(x_{3}\otimes y_{3}\right)=\left(x_{1}x_{2}\right)x_{3}\otimes\left(y_{1}y_{2}\right)y_{3}\\
 & =x_{1}\left(x_{2}x_{3}\right)\otimes y_{1}\left(y_{2}y_{3}\right)=\left(x_{1}\otimes y_{1}\right)\left(x_{2}x_{3}\otimes y_{2}y_{3}\right)\\
 & =\left(x_{1}\otimes y_{1}\right)\left(\left(x_{2}\otimes y_{2}\right)\left(x_{3}\otimes y_{3}\right)\right).
\end{align*}
Furthermore, 
\[
1\cdot x\otimes y=\left(1\cdot x\right)\otimes y=x\otimes y,
\]
\begin{gather*}
\alpha\cdot\left(\beta\cdot x\otimes y\right)=\alpha\cdot\left(\left(\beta\cdot x\right)\otimes y\right)=\left(\alpha\cdot\left(\beta\cdot x\right)\right)\otimes y=\left(\alpha\beta\cdot x\right)\otimes y=\alpha\beta\cdot x\otimes y,
\end{gather*}
\begin{align*}
\lambda\cdot\left(x_{1}\otimes y_{1}\right)\left(x_{2}\otimes y_{2}\right)\\
 & =\lambda\cdot x_{1}x_{2}\otimes y_{1}y_{2}=\left(\lambda\cdot x_{1}x_{2}\right)\otimes y_{1}y_{2}=\left(\left(\lambda\cdot x_{1}\right)x_{2}\right)\otimes y_{1}y_{2}\\
 & =\left(\left(\lambda\cdot x_{1}\right)\otimes y_{1}\right)\left(x_{2}\otimes y_{2}\right)=\left(\lambda\cdot x_{1}\otimes y_{1}\right)\left(x_{2}\otimes y_{2}\right),
\end{align*}
\begin{align*}
\lambda\cdot\left(x_{1}\otimes y_{1}\right)\left(x_{2}\otimes y_{2}\right)\\
 & =\lambda\cdot x_{1}x_{2}\otimes y_{1}y_{2}=\left(\lambda\cdot x_{1}x_{2}\right)\otimes y_{1}y_{2}=x_{1}x_{2}\otimes\left(\lambda\cdot y_{1}y_{2}\right)\\
 & =x_{1}x_{2}\otimes\left(y_{1}\left(\lambda\cdot y_{2}\right)\right)=\left(x_{1}\otimes y_{1}\right)\left(x_{2}\otimes\left(\lambda\cdot y_{2}\right)\right)\\
 & =\left(x_{1}\otimes y_{1}\right)\left(\lambda\cdot x_{2}\otimes y_{2}\right),
\end{align*}
so $X\otimes Y$ is a scalable monoid.
\end{proof}
It follows that if $X,Y,Z$ are scalable monoids over $R$ then $\left(X\otimes Y\right)\otimes Z$
and $X\otimes\left(Y\otimes Z\right)$ are scalable monoids over $R$:
the tensor product of $X\otimes Y$ and $Z$ in the first case and
of $X$ and $Y\otimes Z$ in the second case. It can also be shown
that $\left(X\otimes Y\right)\otimes Z$ and $X\otimes\left(Y\otimes Z\right)$
are isomorphic scalable monoids.

The tensor product can be used to ''glue'' scalable monoids together
in a natural way so as to combine them into more inclusive scalable
monoids. For example, $R\left\llbracket x;y\right\rrbracket $ is
isomorphic to the tensor product $R\left\llbracket x\right\rrbracket \otimes R\left\llbracket y\right\rrbracket $
but not to the direct product $R\left\llbracket x\right\rrbracket \times R\left\llbracket y\right\rrbracket $.

\subsection{\label{subsec:Quotients-of-scalable}Quotients of scalable monoids
by normal submonoids}

In a monoid we have $x\left(yz\right)=\left(xy\right)z$ and $1_{\!X}x=x=x1_{\!X}$,
so a submonoid $\mathfrak{M}$ of a scalable monoid $X$ can act as
a monoid on $X$ by left or right multiplication. In particular, we
can define an action $\pi:\mathfrak{M}\times X\rightarrow X$ by setting
$\pi\left(m,x\right)=mx$ for any $m\in\mathfrak{M}$ and $x\in X$.
This action can be used to define further notions in the same way
that $\sim$ , $\left[x\right]$ and $X/{\sim}$ were defined in terms
of the scaling action $\omega:R\times X\rightarrow X$.
\begin{defn}
\label{def5}Let $X$ be a scalable monoid and $\mathfrak{M}$ a submonoid
of $X$. Then $\sim_{\mathfrak{M}}$ is the relation on $X$ such
that $x\sim_{\mathfrak{M}}y$ if and only if $mx=ny$ for some $m,n\in\mathfrak{M}$. 
\end{defn}

A \emph{normal} submonoid of a scalable monoid $X$ is a submonoid
$\mathfrak{M}$ of $X$ such that $x\mathfrak{M}=\mathfrak{M}x$ for
every $x\in X$. It is clear that if $\mathfrak{M}$ is a \emph{central}
submonoid of $X$, that is, if every element of $\mathfrak{M}$ commutes
with every element of $X$, then $\mathfrak{M}$ is normal, and every
submonoid of a commutative scalable monoid is normal.
\begin{prop}
\label{s3.1-1}If $X$ is a scalable monoid and $\mathfrak{M}$ a
normal submonoid of $X$ then $\sim_{\mathfrak{M}}$ is an equivalence
relation on $X$. 
\end{prop}

\begin{proof}
The relation $\sim_{\mathfrak{M}}$ is reflexive since $1_{\!X}x\sim_{\mathfrak{M}}1_{\!X}x$
for all $x\in X$, symmetric by construction, and transitive because
if $mx=ny$ and $m'y=n'z$ for $x,y,z\in X$ and $m,n,m',n'\in\mathfrak{M}$
then there is some $n_{0}\in\mathfrak{M}$ such that $m'mx=m'ny=n_{0}m'y=n_{0}n'z$,
where $m'm,n_{0}n'\in\mathfrak{M}$.
\end{proof}
\begin{defn}
We denote the equivalence class of $x$ for $\sim_{\mathfrak{M}}$
by $\left[x\right]_{\mathfrak{M}}$, and the set of equivalence classes
$\left\{ \left[x\right]_{\mathfrak{M}}\mid x\in X\right\} $ by $X/\mathfrak{M}$.
\end{defn}

Results analogous to Propositions \ref{s3.2} and \ref{s3.3-1} hold
for scalable monoids with normal submonoids. 
\begin{prop}
\label{s3.2-1-2}If $X$ is a scalable monoid over $R$ and $\mathfrak{M}$
a normal submonoid of $X$ then the relation $\sim_{\mathfrak{M}}$
is a congruence on $X$ with regard to the operations $\left(x,y\right)\mapsto xy$
and $\left(\lambda,x\right)\mapsto\lambda\cdot x$.
\end{prop}

\begin{proof}
If $mx=nx'$ and $m'y=n'y'$ for some $x,x',y,y'\in X$ and $m,n,m',n'\in\mathfrak{M}$
then $(mx)(m'y)=(nx')(n'y')$, so $\left(mm_{0}'\right)\left(xy\right)=\left(nn_{0}'\right)\left(x'y'\right)$.
Hence, if $x\sim_{\mathfrak{M}}x'$ and $y\sim_{\mathfrak{M}}y'$
then $xy\sim_{\mathfrak{M}}x'y'$ since $m_{0}',n_{0}',mm_{0}',nn_{0}'\in\mathfrak{M}$. 

Also, if $mx=nx'$ for some $m,n\in M$ then $\lambda\cdot mx=\lambda\cdot nx'$
for all $\lambda\in R$, so $m\left(\lambda\cdot x\right)=n\left(\lambda\cdot x'\right)$.
Hence, if $x\sim_{\mathfrak{M}}x'$ then $\lambda\cdot x\sim_{\mathfrak{M}}\lambda\cdot x'$.
\end{proof}

In view of Proposition \ref{s3.2-1-2}, we can define operations on
$X/\mathfrak{M}$ as follows.
\begin{defn}
\label{def:214-1}Set $[x]_{\mathfrak{M}}[y]_{\mathfrak{M}}=[xy]_{\mathfrak{M}}$,
$\lambda\cdot\left[x\right]_{\mathfrak{M}}=\left[\lambda\cdot x\right]_{\mathfrak{M}}$
and $1_{X/\mathfrak{M}}=\left[1_{\!X}\right]_{\mathfrak{M}}$ for
any $\left[x\right],\left[y\right]\in X/{\sim}$ and $\lambda\in R$.
\end{defn}

With these definitions, $X/\mathfrak{M}$ is an $\left(R,\boldsymbol{1}\right)$-magma
and the surjective function $\phi_{\mathfrak{M}}:X\rightarrow X/\mathfrak{M}$
defined by $\phi_{\mathfrak{M}}\left(x\right)=\left[x\right]_{\mathfrak{M}}$
satisfies the conditions 
\[
\phi_{\mathfrak{M}}\left(xy\right)=\phi_{\mathfrak{M}}\left(x\right)\phi_{\mathfrak{M}}\left(y\right),\quad\phi_{\mathfrak{M}}\left(\lambda\cdot x\right)=\lambda\cdot\phi_{\mathfrak{M}}\left(x\right),\quad\phi_{\mathfrak{M}}\left(1_{\!X}\right)=1_{X/\mathfrak{M}},
\]
so $\phi_{\mathfrak{M}}$ is a homomorphism of $\left(R,\boldsymbol{1}\right)$-magmas
and hence of scalable monoids%
.
\begin{prop}
\label{p8}If $X$ is a scalable monoid over $R$ and $\mathfrak{M}$
a normal submonoid of $X$ then $X/\mathfrak{M}$ is a\textsf{ }scalable
monoid over $R$ and the function 
\[
\phi_{\mathfrak{M}}:X\rightarrow X/\mathfrak{M},\qquad x\mapsto\left[x\right]_{\mathfrak{M}}
\]
is a surjective homomorphism of scalable monoids.
\end{prop}

Let us now consider the special case where the submonoid $\mathfrak{M}$
of $X$ considered above is a \emph{scalable} submonoid $M$ so that
$x\in M$ implies $\lambda\cdot x\in M$ for every $\lambda\in R$.
\begin{prop}
If $M$ is a normal scalable submonoid of a scalable monoid $X$ over
$R$ then $X/M$ is a trivially scalable monoid over $R$.
\end{prop}

\begin{proof}
If $M$ is a scalable submonoid of $X$ then $\left(\lambda\cdot x\right)\sim_{M}x$
for any $\lambda\in R$ and $x\in X$ since $1_{\!X}\left(\lambda\cdot x\right)=\left(\lambda\cdot1_{\!X}\right)x$,
where $1_{\!X},\lambda\cdot1_{\!X}\in M$. Hence, $\lambda\cdot\left[x\right]_{M}=\left[\lambda\cdot x\right]_{M}=\left[x\right]_{M}$
for any $\lambda\in R$ and $\left[x\right]_{M}\in X/M$.
\end{proof}
\begin{prop}
\label{p28-1}If $M$ is a scalable submonoid of a scalable monoid
$X$ over R and $x,y\in X$ then $x\sim y$ implies $x\sim_{\!M\!}y$.
\end{prop}

\begin{proof}
If if $\alpha\cdot x=\beta\cdot y$ for some $\alpha,\beta\in R$
then 
\[
\left(\alpha\cdot1_{\!X}\right)x=\alpha\cdot1_{\!X}x=\beta\cdot1_{\!X}y=\left(\beta\cdot1_{\!X}\right)y.
\]
This implies the assertion since $\alpha\cdot1_{\!X},\beta\cdot1_{X}\in M$. 
\end{proof}
\begin{prop}
\label{p29-1}Let $X$ be a scalable monoid over R and $x,y\in X$.
Then $R\cdot1_{\!X}$ is a \textcolor{black}{normal} scalable submonoid
of $X$, and $x\sim_{R\cdot1_{\!X}}\!y$ implies $x\sim y$.
\end{prop}

\begin{proof}
Consider any $\alpha,\beta,\lambda\in R,x\in X$. $R\cdot1_{\!X}$
is (1) a submonoid of $X$ since $1_{\!X}=1\cdot1_{\!X}$ and $\left(\alpha\cdot1_{\!X}\right)\left(\beta\cdot1_{\!X}\right)=\alpha\beta\cdot1_{\!X}$
where $1,\alpha\beta\in R$, (2) a scalable submonoid of $X$ since
$\lambda\cdot\left(\alpha\cdot1_{\!X}\right)=\lambda\alpha\cdot1_{\!X}$
where $\lambda\alpha\in R$, and (3) a \textcolor{black}{normal} submonoid
of $X$ since $\left(\alpha\cdot1_{\!X}\right)x=\alpha\cdot1_{\!X}x=\alpha\cdot x1_{\!X}=x\left(\alpha\cdot1_{\!X}\right)$.
Also, if $x\sim_{R\cdot1_{\!X}}y$ then 
\[
\alpha\cdot1_{\!X}x=\left(\alpha\cdot1_{\!X}\right)x=\left(\beta\cdot1_{\!X}\right)y=\beta\cdot1_{\!X}y
\]
 for some $\alpha,\beta\in R$, so $x\sim_{R\cdot1_{\!X}}\!y$ implies
$x\sim y$. 
\end{proof}
It follows from Propositions \ref{p28-1} and \ref{p29-1} that $x\sim_{M}y$
generalizes $x\sim y$.
\begin{cor}
Let $X$ be a scalable monoid over $R$ and $x,y\in X$. Then $R\cdot1_{\!X}$
is a \textcolor{black}{normal} scalable submonoid of $X$, and $x\sim_{R\cdot1_{\!X}}\!y$
if and only if $x\sim y$. 
\end{cor}

For example, $\kappa\cdot\lambda x_{1}^{k_{1}}\ldots x_{n}^{k_{n}}=\left(\kappa\cdot1x_{1}^{0}\ldots x_{n}^{0}\right)\left(\lambda x_{1}^{k_{1}}\ldots x_{n}^{k_{n}}\right)$
for any $\kappa\in R$, so $R\left\llbracket x_{1};\ldots;x_{n}\right\rrbracket /{\sim}$
and $R\left\llbracket x_{1};\ldots;x_{n}\right\rrbracket /\left(R\cdot1_{\!X}\right)$
are isomorphic monoids.

\subsection{\label{subsec:Commensurability-classes-with}A non-trivial orbitoid
with a unit element is a free module of rank 1}

Recall the principle that magnitudes of the same kind can be added
and subtracted, whereas magnitudes of different kinds cannot be combined
by these operations. Also recall the idea that a quantity can be represented
by a ''unit'' and a number (measure) specifying ''{[}how many{]}
times the {[}unit{]} is to be taken in order to make up'' that quantity
\cite[p. 41]{MAXW}. As shown below, there is a connection between
these two notions.

Specifically, it may happen that $R\cdot u\supseteq\left[u\right]$
for some $u\in\left[u\right]$, and if in addition a natural uniqueness
condition is satisfied we may regard $u$ as a unit of measurement
for $\left[u\right]$. If such a unit exists then a sum of elements
of $\left[u\right]$ can be defined by the construction described
in Definition \ref{d3.3} below.
\begin{defn}
\label{d26}Let $\mathsf{C}$ be an orbitoid in a scalable monoid
over $R$. A \emph{generating element} for $\mathsf{C}$ is some $u\in\mathsf{C}$
such that for every $x\in\mathsf{C}$ there is some $\lambda\in R$
such that $x=\lambda\cdot u$. A \emph{unit element for $\mathsf{C}$}
is a generating element $u$ for $\mathsf{C}$ such that if $\lambda\cdot u=\lambda'\cdot u$
then $\lambda=\lambda'$. 
\end{defn}

By this definition, if $u$ is a generating element for $\mathsf{C}=\left[u\right]$
then $R\cdot u\supseteq\mathsf{C}$. On the other hand, $\lambda\cdot u\sim u$
for any $\lambda\in R$, so $\lambda\cdot u\in\left[u\right]$ for
any $\lambda\in R$, so $R\cdot u\subseteq\left[u\right]$. Thus,
actually $R\cdot u=\mathsf{C}$.

We now need to consider zero elements in scalable monoids.
\begin{prop}
Let $X$ be a scalable monoid over $R$. For every $\mathsf{C}\in X/{\sim}$
there is a unique $0_{\mathsf{C}}\in\mathsf{C}$ such that $0_{\mathsf{C}}=0\cdot x$
for all $x\in\mathsf{C}$.
\end{prop}

\begin{proof}
Only uniqueness needs to be proved. If $x,y\in\mathsf{C}$ then $\alpha\cdot x=\beta\cdot y$
for some $\alpha,\beta\in R$, so $0\cdot x=0\alpha\cdot x=0\cdot\left(\alpha\cdot x\right)=0\cdot\left(\beta\cdot y\right)=0\beta\cdot y=0\cdot y$.
\end{proof}
We call $0_{\mathsf{C}}$ the \emph{zero element} of $\mathsf{C}$;
note that distinct orbitoids have distinct zero elements since $0_{\mathsf{A}}=0_{\mathsf{B}}$
implies $\mathsf{A}=\left[0_{\mathsf{A}}\right]=\left[0_{\mathsf{B}}\right]=\mathsf{B}$.
It is clear that $\lambda\cdot0_{\mathsf{C}}=0_{\mathsf{C}}$ for
all $\lambda\in R$, and that $0_{\left[x\right]}y=0_{\left[xy\right]}$
and $y0_{\left[x\right]}=0_{\left[yx\right]}$ for all $x,y\in X$. 

A \emph{trivial} orbitoid is an orbitoid $\mathsf{C}=\left\{ 0_{\mathsf{C}}\right\} $.%

We now turn to a lemma and a definition leading to Proposition \ref{p5}.
\begin{lem}
\label{L22}Let X be a scalable monoid over $R$. If $u$ and $u'$
are unit elements for $\mathsf{C}\in X/{\sim}$, $\rho,\sigma,\rho',\sigma'\in R$,
$\rho\cdot u=\rho'\cdot u'$ and $\sigma\cdot u=\sigma'\cdot u'$
then $\left(\rho+\sigma\right)\cdot u=\left(\rho'+\sigma'\right)\cdot u'$. 
\end{lem}

\begin{proof}
As $u'\in\mathsf{C}$, there is a unique $\tau\in R$ such that $u'=\tau\cdot u$.
Thus, 
\begin{gather*}
\rho\cdot u=\rho'\cdot u'=\rho'\cdot\left(\tau\cdot u\right)=\rho'\tau\cdot u,\\
\sigma\cdot u=\sigma'\cdot u'=\sigma'\cdot\left(\tau\cdot u\right)=\sigma'\tau\cdot u,\\
\left(\rho'+\sigma'\right)\cdot u'=\left(\rho'+\sigma'\right)\cdot\left(\tau\cdot u\right)=\left(\rho'+\sigma'\right)\tau\cdot u=\left(\rho'\tau+\sigma'\tau\right)\cdot u,
\end{gather*}
so $\left(\rho'+\sigma'\right)\cdot u'=\left(\rho'\tau+\sigma'\tau\right)\cdot u=\left(\rho+\sigma\right)\cdot u$
since $\rho=\rho'\tau$ and $\sigma=\sigma'\tau$.
\end{proof}
Hence, the sum of two elements of a scalable monoid can be defined
as follows.
\begin{defn}
\label{d3.3}Let $X$ be a scalable monoid over $R$, and let $u$
be a unit element for $\mathsf{C}\in X/{\sim}$. If $x=\rho\cdot u$
and $y=\sigma\cdot u$, where $\rho,\sigma\in R$, we set 
\[
x+y=\left(\rho+\sigma\right)\cdot u.
\]
\end{defn}

The sum $x+y$ is given by Definition \ref{d3.3} if and only if $x$
and $y$ are commensurable and their orbitoid has a unit element.
This suggests again that the concept of commensurability introduced
in Definition \ref{d26} can be used to define the ancient Greek notion
of magnitudes of the same kind, and to clarify the modern notion of
quantities of the same kind.

It follows immediately from Definition \ref{d3.3} that 
\[
\left(x+y\right)+z=x+\left(y+z\right),\qquad x+y=y+x
\]
for all $x,y,z\in\mathsf{C}$, and that 
\[
x+0_{\mathsf{C}}=0_{\mathsf{C}}+x=x
\]
 for any $x\in\mathsf{C}$ since $0_{\mathsf{C}}=0\cdot u$. 

If $x=\rho\cdot u$ so that $\lambda\cdot x=\lambda\rho\cdot u$ and
$\kappa\cdot x=\kappa\rho\cdot u$ then 
\begin{gather*}
\left(\lambda+\kappa\right)\cdot x=\left(\lambda+\kappa\right)\cdot\left(\rho\cdot u\right)=\left(\lambda+\kappa\right)\rho\cdot u=\left(\lambda\rho+\kappa\rho\right)\cdot u=\lambda\cdot x+\kappa\cdot x,
\end{gather*}
and if $x=\rho\cdot u$ and $y=\sigma\cdot u$ so that $\lambda\cdot x=\lambda\rho\cdot u$
and $\lambda\cdot y=\lambda\sigma\cdot u$ then 
\begin{gather*}
\lambda\cdot\left(x+y\right)=\lambda\cdot\left(\left(\rho+\sigma\right)\cdot u\right)=\lambda\left(\rho+\sigma\right)\cdot u=\left(\lambda\rho+\lambda\sigma\right)\cdot u=\lambda\cdot x+\lambda\cdot y.
\end{gather*}

A unital ring $R$ has a unique additive inverse $-1$ of $1\in R$,
and we set 
\[
-x=\left(-1\right)\cdot x
\]
for all $x\in X$. If $\mathsf{C}$ has a unit  element $u$ and $x=\rho\cdot u$
for some $\rho\in R$ then $-x=\left(-1\right)\cdot\left(\rho\cdot u\right)=\left(-\rho\right)\cdot u$,
and using this fact it is easy to verify that
\[
x+\left(-x\right)=-x+x=0_{\mathsf{C}}.
\]
As usual, we may write $x+\left(-y\right)$ as $x-y$, and thus $x+\left(-x\right)$
as $x-x$.

While a trivial orbitoid is a zero module $\left\{ 0_{\mathsf{C}}\right\} $
with $0_{\mathsf{C}}+0_{\mathsf{C}}=0_{\mathsf{C}}$ and $\lambda\cdot0_{\mathsf{C}}=0_{\mathsf{C}}$
for all $\lambda\in R$, a non-trivial orbitoid with a unit element
is a well-behaved module.
\begin{prop}
\label{p5}Let $X$ be a scalable monoid over $R$. If $\mathsf{C}\in X/{\sim}$
is a non-trivial orbitoid with a unit element then $R$ is \textup{\emph{a
non-trivial commutative ring, and}} $\mathsf{C}$, with appropriate
definitions of $x+y$ and $\lambda\cdot x$, is a free module of rank
1 over $R$\textup{\emph{.}}
\end{prop}

\begin{proof}
Let $u$ be a unit element for $\mathsf{C}$. If $0_{\mathsf{C}}\neq x\in\mathsf{C}$,
$0_{\mathsf{C}}=\lambda\cdot u$ and $x=\kappa\cdot u$ for some $\lambda,\kappa\in R$
then $\lambda\neq\kappa,$ so $R$ is non-trivial.%
{} We also have $\alpha\beta\cdot u=\beta\alpha\cdot u$ for any $\alpha,\beta\in R$
by Lemma \ref{thm:lem1}, so $\alpha\beta=\beta\alpha$ since $u$
is a unit element.

We have seen that $\mathsf{C}$ is a module with addition given by
Definition \ref{d3.3} and scalar multiplication inherited from the
scalar multiplication in $X$. Also, if $u$ is a unit element for
$\mathsf{C}$ then $\left\{ u\right\} $ is a basis for $\mathsf{C}$,
and $R$ has the invariant basis number property since it is non-trivial
and commutative \cite{RICH}.
\end{proof}
Thus, if every orbitoid $\mathsf{C}\in X/{\sim}$ contains a non-zero
unit element for $\mathsf{C}$ then $X$ is the union of disjoint
isomorphic free modules of rank 1 over a non-trivial commutative ring,
a result that may be compared to definitions of systems of quantities
in terms of unions of one-dimensional vector spaces by Quade \cite{QUAD}
and Raposo \cite{RAP}.

Recall that identities corresponding to $\left(\lambda+\kappa\right)\cdot x=\lambda\cdot x+\kappa\cdot x,$
$\lambda\cdot\left(x+y\right)=\lambda\cdot x+\lambda\cdot y$ and
$\lambda\cdot\left(\kappa\!\cdot\!x\right)=\lambda\kappa\cdot x$
were proved in Propositions 1\textendash 3 in Book V of the \emph{Elements,}
so rudiments of Proposition \ref{p5} were present already in the
Greek theory of magnitudes.

\subsection{\label{s25}Scalable monoids with sets of unit elements }

In this section, we build on the discussion in the two previous sections
about unit elements and quotients of scalable monoids by normal monoids. 
\begin{defn}
\label{d8}A \emph{dense} set of elements of a scalable monoid $X$
is a set $U$ of elements of $X$ such that for every $x\in X$ there
is some $u\in U$ such that $u\sim x$. A \emph{sparse} set of elements
of $X$ is a set $U$ of elements of $X$ such that $u\sim v$ implies
$u=v$ for any $u,v\in U$. A \emph{closed} set of elements of $X$
is a set $U$ of elements of $X$ such that if $u,v\in U$ then $uv\in U$. 

We call a (dense) sparse set of unit elements of $X$ a \emph{(complete)
system of unit elements }for $X$.
\end{defn}

\begin{defn}
\label{d9}A \emph{distributive} scalable monoid $X$ is a scalable
monoid such that for all $\mathsf{A},\mathsf{B}\in X/{\sim}$ we have
\[
\left(x+y\right)z=xz+yz,\qquad z\left(x+y\right)=zx+zy,
\]
for all $x,y\in\mathsf{A}$ and all $z\in\mathsf{B}$.
\end{defn}

\begin{prop}
\label{p11}Let $X$ be a scalable monoid. If $X$ is equipped with
a dense closed set of unit elements $U$ then $X$ is a distributive
scalable monoid.
\end{prop}

\begin{proof}
For all $x,y\in\mathsf{A}$ and $z\in\mathsf{B}$ there are $u,v\in U$
such that $\left[x\right]=\left[y\right]=\left[u\right]$ and $\left[z\right]=\left[v\right]$
since $U$ is dense in $X$, so $x=\rho\cdot u$, $y=\sigma\cdot u$
and $z=\tau\cdot v$ for some $\rho,\sigma,\tau\in R$, so $xz=\rho\tau\cdot uv,$
$yz=\sigma\tau\cdot uv$, $zx=\tau\rho\cdot vu$ and $zy=\tau\sigma\cdot vu$,
so
\begin{gather*}
\left(x+y\right)z=\left(\left(\rho+\sigma\right)\cdot u\right)\left(\tau\cdot v\right)=\left(\rho+\sigma\right)\tau\cdot uv=\left(\rho\tau+\sigma\tau\right)\cdot uv=xz+yz,\\
z\left(x+y\right)=\left(\tau\cdot v\right)\left(\left(\rho+\sigma\right)\cdot u\right)=\tau\left(\rho+\sigma\right)\cdot vu=\left(\tau\rho+\tau\sigma\right)\cdot vu=zx+zy,
\end{gather*}
using the fact that $uv$ and $vu$ are unit elements since $U$ is
closed.
\end{proof}
We also define a natural notion which is fundamental in metrology.
\begin{defn}
\label{d12}A \emph{coherent }system of unit elements for $X$ is
a submonoid of $X$ which is a complete system of unit elements for
$X$.
\end{defn}

Recall that if $T$ is a normal submonoid of a scalable monoid $X$
then $X/T$ is a scalable monoid by Proposition \ref{p8}. It is proved
in Appendix A that if $S\supseteq T$ is a coherent system of unit
elements for $X$ then $S/T$ is a coherent system of unit elements
for $X/T$.%

\section{\label{sec:Quantity-spaces}Quantity spaces}

In this section, we specialize scalable monoids in order to obtain
a mathematical model suitable for calculation with quantities, a quantity
space.

The formal definition of a quantity space is given in Section \ref{s21},
and some basic facts about quantity spaces are presented in Section
\ref{s23}. Coherent systems of unit quantities for quantity spaces
are discussed in Section \ref{s24}. The notion of a measure of a
quantity is formally defined in Section \ref{s25-1}, and ways in
which measures serve as proxies for quantities are described. In Section
\ref{s26}, we show that the monoid of dimensions $Q/{\sim}$ corresponding
to a quantity space $Q$ is a free abelian group and derive some related
results. 

\subsection{\label{s21}Canonical construction and main definition}

It is possible to give an abstract definition of scalable monoids
of the form $R\left\llbracket x_{1};\ldots;x_{n}\right\rrbracket $
(Example \ref{ex1}). Let $X$ be a commutative scalable monoid over
a commutative ring $R$. A \emph{finite scalable-monoid basis} for\emph{
$X$} is a finite set $\left\{ e_{1},\ldots,e_{n}\right\} $ of elements
of $X$ such that every $x\in X$ has a unique expansion 
\[
x=\mu\cdot\prod_{i=1}^{n}\nolimits e_{i}^{_{_{k_{i}}}},
\]
where $\mu\in R$ and $k_{i}$ are non-negative integers.%
. In abstract terms, $R\left\llbracket x_{1};\ldots;x_{n}\right\rrbracket $
is a %
commutative scalable monoid $X$ over a commutative ring, such that
there exists a finite scalable-monoid basis for $X$.

Now, consider instead the set $K\left\llbracket x_{1},x_{1}^{-1};\ldots;x_{n},x_{n}^{-1}\right\rrbracket $
of all Laurent monomials of the form $\lambda x_{1}^{k_{1}}\ldots x_{n}^{k_{n}}$,
where $\lambda\in K$, $K$ is a field, $x_{1},\ldots,x_{n}$ are
uninterpreted symbols and $k_{1},\ldots,k_{n}$ are integers, together
with essentially the same operations as in Example \ref{ex1}, namely
\begin{equation}
\begin{cases}
\left(\lambda x_{1}^{j_{1}}\ldots x_{n}^{j_{n}}\right)\left(\kappa x_{1}^{k_{1}}\ldots x_{n}^{k_{n}}\right)=\left(\lambda\kappa\right)x_{1}^{\left(j_{1}+k_{1}\right)}\ldots x_{n}^{\left(j_{n}+k_{n}\right)},\\
\alpha\cdot\lambda x_{1}^{k_{1}}\ldots x_{n}^{k_{n}}=\left(\alpha\lambda\right)x_{1}^{k_{1}}\ldots x_{n}^{k_{n}},\qquad1_{K\left\llbracket x_{1},x_{1}^{-1};\ldots;x_{n},x_{n}^{-1}\right\rrbracket }=1x_{1}^{0}\ldots x_{n}^{0}.
\end{cases}\label{rep}
\end{equation}
Any such $K\left\llbracket x_{1},x_{1}^{-1};\ldots;x_{n},x_{n}^{-1}\right\rrbracket $
is likewise a scalable monoid, which can again be characterized abstractly.
\begin{defn}
\label{def:qs-bas}Let $Q$ be a commutative scalable monoid over
a field $K$\emph{.} A \emph{finite quantity-space basis} for\emph{
$Q$} is a finite set $\left\{ e_{1},\ldots,e_{n}\right\} $ of invertible
elements of $Q$ such that every $x\in Q$ has a unique expansion
\[
x=\mu\cdot\prod_{i=1}^{n}\nolimits e_{i}^{_{_{k_{i}}}},
\]
where $\mu\in K$ and $k_{i}$ are integers.%
{} 
\end{defn}

{} It is easy to show that any commutative scalable monoid $Q$ over
a field, such that there exists a finite quantity-space basis for
$Q$, %
can be represented by some $K\left\llbracket x_{1},x_{1}^{-1};\ldots;x_{n},x_{n}^{-1}\right\rrbracket $
\cite{JON3}. %
On the other hand, we have the following abstract characterization
of this kind of scalable monoid, corresponding to a finitely generated
free abelian group and well suited for applications in theoretical
metrology, dimensional analysis etc. 
\begin{defn}
\label{d2.4-1}A \emph{finitely generated quantity space} is a commutative
scalable monoid $Q$ over a field, such that there exists a\emph{
}finite\emph{ }quantity-space basis\emph{ }for\emph{ $Q$.}
\end{defn}

Although finitely generated quantity spaces can be readily generalized
to quantity spaces with infinite bases, only the finite case will
be considered here. Below, ''basis'' and ''quantity space'' will
be understood to mean ''finite quantity-space basis'' and ''finitely
generated quantity space'', respectively. 

Elements of a quantity space are called \emph{quantities,} unit elements
are called \emph{unit quantities}, and orbitoids in a quantity space
are called \emph{dimensions}.

Note that $K\left\llbracket x_{1};\ldots;x_{n}\right\rrbracket $,
where $K$ is a field, is not a quantity space, so a commutative scalable
monoid over a field is not necessarily a quantity space: the relationship
between a scalable monoid and a quantity space is not as close as
that between a module and a vector space.

\subsection{\label{s23}Some basic properties of quantity spaces}
\begin{prop}
\label{p19a}Let $Q$ be a quantity space over $K$ with a basis $\left\{ e_{1},\ldots,e_{n}\right\} $
and $x,y\in Q$. Then
\begin{enumerate}
\item $1_{Q}=1\cdot\prod_{i=1}^{n}e_{i}^{0}$\emph{;}
\item if $x=\mu\cdot\prod_{i=1}^{n}e_{i}^{_{_{k_{i}}}}$ and $y=\nu\cdot\prod_{i=1}^{n}e_{i}^{\ell_{i}}$
then $xy=\mu\nu\cdot\prod_{i=1}^{n}e_{i}^{_{\left(k_{i}+\ell_{i}\right)}}$\emph{;}
\item if $x=\mu\cdot\prod_{i=1}^{n}e_{i}^{_{_{k_{i}}}}$ and $\mu\neq0$
then $x^{-1}$ exists and $x^{-1}=\mu^{-1}\cdot\prod_{i=1}^{n}e_{i}^{_{_{-k_{i}}}}$.
\end{enumerate}
\end{prop}

\begin{proof}
(1) Note that $e_{i}^{0}=1_{Q}$ for all $e_{i}$. (2) This follows
from Lemma \ref{thm:lem1} and the fact that $Q$ is commutative.
(3) Note that $\mu^{-1}\cdot\prod_{i=1}^{n}e_{i}^{_{_{-k_{i}}}}\in Q$
since $\mu\in K$ and $e_{1},\ldots,e_{n}\in Q$, so this follows
from (1) and (2).
\end{proof}
\begin{prop}
\label{p20-1}Let $Q$ be a quantity space over $K$ with a basis
$\left\{ e_{1},\ldots e_{n}\right\} $ and $x=\mu\cdot\prod_{i=1}^{n}e_{i}^{k_{i}}$.
Then the following conditions are equivalent:
\begin{enumerate}
\item $x$ is a non-zero quantity\emph{;}
\item $\mu\neq0$\emph{;}
\item x is invertible.
\end{enumerate}
\end{prop}

\begin{proof}
$\left(1\right)\Longleftrightarrow\left(2\right)$. Note that $0\cdot x=0\cdot\left(\mu\cdot\prod_{i=1}^{n}e_{i}^{k_{i}}\right)=0\cdot\prod_{i=1}^{n}e_{i}^{k_{i}}$.
Thus, if $\mu=0$ then $0\cdot x=x$, so $x$ is a zero quantity.
Conversely, if $0\cdot x=x$ then $0\mu\cdot\prod_{i=1}^{n}e_{i}^{k_{i}}=0\cdot\left(\mu\cdot\prod_{i=1}^{n}e_{i}^{k_{i}}\right)=\mu\cdot\prod_{i=1}^{n}e_{i}^{k_{i}}$,
so $\mu=0$ since the expansion of $x$ is unique.

$\left(2\right)\Longleftrightarrow\left(3\right)$. If $\mu\neq0$
then $x$ has an inverse by Proposition \ref{p19a}. Conversely, if
$\mu=0$ then $\mu\nu=0\neq1$ for all $\nu\in K$, so $x$ does not
have an inverse $\nu\cdot\prod_{i=1}^{n}e_{i}^{\ell_{i}}$. 
\end{proof}
Thus, $1_{Q}$ is a non-zero quantity, and all elements of a basis
are non-zero quantities. Also, it follows from Proposition \ref{p20-1}
that $Q$ has no zero divisors.
\begin{cor}
\label{cor35}The product of non-zero quantities is a non-zero-quantity,
and the non-zero quantities in a dimension $\mathsf{C}$ form an abelian
group.
\end{cor}

\begin{lem}
\label{s3.3}Let $Q$ be a quantity space over $K$ with a basis $\left\{ e_{1},\ldots e_{n}\right\} $,
and consider $x=\mu\cdot\prod_{i=1}^{n}e_{i}^{_{_{k_{i}}}}$ and $y=\nu\cdot\prod_{i=1}^{n}e_{i}^{_{_{\ell_{i}}}}$.
The following conditions are equivalent: 
\begin{enumerate}
\item $x\sim y$, or equivalently $\mu\cdot\prod_{i=1}^{n}e_{i}^{_{_{k_{i}}}}\sim\nu\cdot\prod_{i=1}^{n}e_{i}^{_{_{\ell_{i}}}}$\emph{;} 
\item $k_{i}=\ell_{i}$ for $i=1,\ldots,n$\emph{;} 
\item $\prod_{i=1}^{n}e_{i}^{_{_{k_{i}}}}=\prod_{i=1}^{n}e_{i}^{_{_{\ell_{i}}}}$\emph{;}
\item $\nu\cdot x=\mu\cdot y$, or equivalently $\nu\cdot\left(\mu\cdot\prod_{i=1}^{n}e_{i}^{_{_{k_{i}}}}\right)=\mu\cdot\left(\nu\cdot\prod_{i=1}^{n}e_{i}^{_{_{\ell_{i}}}}\right)$. 
\end{enumerate}
\end{lem}

\begin{proof}
Implications $(2)\Longrightarrow(3)$ and $(4)\Longrightarrow(1)$
are trivial, while $(3)\Longrightarrow(4)$ follows from Lemma \ref{thm:lem1}.
To prove $(1)\Longrightarrow(2)$, note that if $x\sim y$ then 
\[
\alpha\mu\cdot\prod_{i=1}^{n}\nolimits e_{i}^{_{_{k_{i}}}}=\alpha\cdot\left(\mu\cdot\prod_{i=1}^{n}\nolimits e_{i}^{_{_{k_{i}}}}\right)=\beta\cdot\left(\nu\cdot\prod_{i=1}^{n}\nolimits e_{i}^{_{_{\ell_{i}}}}\right)=\beta\nu\cdot\prod_{i=1}^{n}\nolimits e_{i}^{_{_{\ell_{i}}}}
\]
for some $\alpha,\beta\in K$, so $k_{i}=\ell_{i}$ for $i=1,\ldots,n$
because of the uniqueness of the expansion of $\alpha\cdot x$.%
\end{proof}
It follows immediately from Lemma \ref{s3.3} that if not $k_{i}=\ell_{i}$
for $i=1,\ldots,n$ then $x\neq y$ since $x\nsim y$; this is the
essence of the principle of dimensional homogeneity formulated by
Fourier \cite{FOUR}.

\subsection{\label{s24}Quantity spaces and unit quantities}
\begin{prop}
\label{p22}If $Q$ is a quantity space then every non-zero $u\in Q$
is a unit quantity for $\left[u\right]$.
\end{prop}

\begin{proof}
Let $\left\{ e_{1},\ldots,e_{n}\right\} $ be a basis for $Q$ and
set $u=\mu\cdot\prod_{i=1}^{n}e_{i}^{_{_{k_{i}}}}$, $x=\nu\cdot\prod_{i=1}^{n}e_{i}^{_{_{\ell_{i}}}}$.
Then $\mu\neq0$ by Proposition \ref{p20-1}, and if $x\sim u$ then
$\mu\cdot x=\nu\cdot u$ by Lemma \ref{s3.3}, so $x=\mu^{-1}\mu\cdot x=\mu^{-1}\cdot\left(\mu\cdot x\right)=\mu^{-1}\cdot\left(\nu\cdot u\right)=\mu^{-1}\nu\cdot u$.
Also, if $\lambda\cdot u=\lambda'\cdot u$ then $\lambda\mu\cdot\prod_{i=1}^{n}e_{i}^{k_{i}}=\lambda\cdot u=\lambda'\cdot u=\lambda'\mu\cdot\prod_{i=1}^{n}e_{i}^{k_{i}}$,
so $\lambda\mu=\lambda'\mu$ since the expansion of $\lambda\cdot u$
is unique%
, so $\lambda=\lambda'$ since $\mu\neq0$. 
\end{proof}
\begin{prop}
\label{preP39}If $Q$ is a quantity space then every $\mathsf{C}\in Q/{\sim}$
contains a non-zero unit quantity.
\end{prop}

\begin{proof}
If $x=\mu\cdot\prod_{i=1}^{n}e_{i}^{_{_{k_{i}}}}\in\mathsf{C}$ then
$u=1\cdot\sum_{i=1}^{n}e_{i}^{k_{i}}$ is non-zero by Proposition
\ref{p20-1} and $u\in\mathsf{C}$ by Lemma \ref{s3.3}, so $u$ is
a unit quantity for $\mathsf{C}$ by Proposition \ref{p22}.
\end{proof}
\begin{prop}
\label{p38-1}Let $Q$ be a quantity space over $K$. Then $\mathsf{C}\in Q/{\sim}$,
with $x+y$ and $\lambda\cdot x$ appropriately defined, is a one-dimensional
vector space over $K$. 
\end{prop}

\begin{proof}
$\mathsf{C}$ is a free module of rank 1 over the field $K$ by Propositions
\ref{p5} and \ref{preP39}.
\end{proof}
\begin{prop}
\label{p26-1}Let $Q$ be a quantity space over $K$ with a basis
$E=\left\{ e_{1},\ldots,e_{n}\right\} $. The subset 
\[
U=\left\{ 1\cdot\prod_{i=1}^{n}\nolimits e_{i}^{k_{i}}\mid k_{i}\in\mathbb{Z}\right\} 
\]
 of $Q$ is a coherent system of unit quantities for $Q$.
\end{prop}

\begin{proof}
By Proposition \ref{p20-1}, all elements of $U$ are non-zero and
hence unit quantities by Proposition \ref{p22}. Also, $U$ is dense
in $Q$ since it follows from $x=\mu\cdot\prod_{i=1}^{n}e_{i}^{_{_{k_{i}}}}$
that $1\cdot x=\mu\cdot\left(1\cdot\prod_{i=1}^{n}e_{i}^{k_{i}}\right)$.
Finally, if $u=1\cdot\prod_{i=1}^{n}e_{i}^{k_{i}}\sim1\cdot\prod_{i=1}^{n}e_{i}^{_{_{\ell_{i}}}}=v$
then $\prod_{i=1}^{n}e_{i}^{k_{i}}=\prod_{i=1}^{n}e_{i}^{_{_{\ell_{i}}}}$
by Lemma \ref{s3.3}, so $u=v$, meaning that $U$ is sparse in $Q$.

It remains to prove that $U$ is a monoid. Clearly, $1_{Q}\in U$
since $1_{Q}=1\cdot\prod_{i=1}^{n}e_{i}^{0}$, and we have
\[
\left(1\cdot\prod_{i=1}^{n}\nolimits e_{i}^{_{_{k_{i}}}}\right)\left(1\cdot\prod_{i=1}^{n}\nolimits e_{i}^{\ell_{i}}\right)=1\cdot\prod_{i=1}^{n}\nolimits e_{i}^{\left(k_{i}+\ell_{i}\right)},
\]
so if $u,v\in U$ then $uv\in U$. Thus, $U$ is a submonoid of $Q$.
\end{proof}
In other words, every basis can be extended to a coherent system of
unit quantities, consisting of quantities that are expressed as products
of basis quantities and their inverses. As a direct consequence, we
have the following result.
\begin{prop}
\label{p310}If $Q$ is a quantity space over $K$ then $Q$ is distributive. 
\end{prop}

\begin{proof}
The assertion follows from Propositions \ref{p26-1} and \ref{p11}.
\end{proof}

\subsection{\label{s25-1}Measures of quantities}
\begin{defn}
\label{d2.5}Let $Q$ be a quantity space over $K$ with a basis $E=\left\{ e_{1},\ldots,e_{n}\right\} $.
The uniquely determined scalar $\mu\in K$ in the expansion 
\[
x=\mu\cdot\prod_{i=1}^{n}\nolimits e_{i}^{k_{i}}
\]
is called the \emph{measure} of $x$ relative to \textbf{$E$} and
will be denoted by $\mu_{E}\left(x\right)$. 
\end{defn}

For example, $1_{Q}=1\cdot\prod_{i=1}^{n}e_{i}^{0}$ for any $E$,
so we have the following simple but useful fact.
\begin{prop}
\label{p25}If $Q$ is a quantity space over $K$ then $\mu_{E}\left(1_{Q}\right)=1$
for any basis $E$ for $Q$.
\end{prop}

Relative to a fixed basis, measures of quantities can be used as proxies
for the quantities themselves. 
\begin{prop}
\label{p38}Let Q be a quantity space over $K$ with a basis $E=\left\{ e_{1},\ldots,e_{n}\right\} $.
Then
\begin{enumerate}
\item $\mu_{E}\left(xy\right)=\mu_{E}\left(x\right)\mu_{E}\left(y\right)$
for all $x,y\in Q$\emph{;}
\item $x^{-1}$ exists and $\mu_{E}\left(x^{-1}\right)=\mu_{E}\left(x\right)^{-1}$
for all $x\in Q$ such that $\mu_{E}\left(x\right)\neq0$\emph{;}
\item $\mu_{E}\left(\lambda\cdot x\right)=\lambda\,\mu_{E}\left(x\right)$
for all $\lambda\in K$ and $x\in Q$\emph{;}
\item $\mu_{E}\left(x+y\right)=\mu_{E}\left(x\right)+\mu_{E}\left(y\right)$
for all $x,y\in X$ such that $x\sim y$.
\end{enumerate}
\end{prop}

\begin{proof}
(1) This follows immediately from Proposition \ref{p19a}(2). (2)
This follows similarly from Proposition \ref{p19a}(3). (3) If $x=\mu_{E}\left(x\right)\cdot\prod_{i=1}^{n}e_{i}^{_{_{k_{i}}}}$
then $\lambda\cdot x=\lambda\cdot\left(\mu_{E}\left(x\right)\cdot\prod_{i=1}^{n}e_{i}^{_{_{k_{i}}}}\right)=\lambda\mu_{E}\left(x\right)\cdot\prod_{i=1}^{n}e_{i}^{_{_{k_{i}}}}$.
(4) If $x=\mu_{E}\left(x\right)\cdot\prod_{i=1}^{n}e_{i}^{k_{i}}$
and $y=\mu_{E}\left(y\right)\cdot\prod_{i=1}^{n}e_{i}^{\ell{}_{i}}$
are the expansions of $x$ and $y$ %
then $\prod_{i=1}^{n}e_{i}^{k_{i}}=\prod_{i=1}^{n}e_{i}^{\ell{}_{i}}$
by Lemma \ref{s3.3}. As $\prod_{i=1}^{n}e_{i}^{k_{i}}$ is non-zero,
and thus a unit quantity for $\left[\prod_{i=1}^{n}e_{i}^{k_{i}}\right]$
by Proposition \ref{p22}, we have $x+y=\left(\mu_{E}\left(x\right)+\mu_{E}\left(y\right)\right)\cdot\prod_{i=1}^{n}\nolimits e_{i}^{k_{i}}$
by Definition \ref{d3.3}.
\end{proof}
\begin{prop}
Let $Q$ be a quantity space over $K$. If $E=\left\{ e_{1},\ldots,e_{n}\right\} $
is a basis for $Q$ and $x=\mu_{E}\left(x\right)\cdot\prod_{i=1}^{n}e_{i}^{k_{i}}$
then $E'=\left\{ \lambda_{1}\cdot e_{1},\ldots,\lambda_{n}\cdot e_{n}\right\} $,
where $\lambda_{i}\neq0$, is a basis for $Q$ and $x=\mu_{E'}\left(x\right)\cdot\prod_{i=1}^{n}\left(\lambda_{i}\cdot e_{i}\right)^{k_{i}}$,
where $\mu_{E'}\left(x\right)=\prod_{i=1}^{n}\lambda_{i}^{-k_{i}}\mu_{E}\left(x\right)$.
\end{prop}

\begin{proof}
We have
\begin{align*}
x & =\mu_{E}\left(x\right)\cdot\prod_{i=1}^{n}\nolimits e_{i}^{k_{i}}=\mu_{E}\left(x\right)\cdot\prod_{i=1}^{n}\nolimits\left(\lambda_{i}^{-1}\cdot\left(\lambda_{i}\cdot e_{i}\right)\right)^{k_{i}}\\
 & =\mu_{E}\left(x\right)\cdot\left(\prod_{i=1}^{n}\nolimits\lambda_{i}^{-k_{i}}\cdot\prod_{i=1}^{n}\nolimits\left(\lambda_{i}\cdot e_{i}\right)^{k_{i}}\right)=\mu_{E}\left(x\right)\prod_{i=1}^{n}\nolimits\lambda_{i}^{-k_{i}}\cdot\prod_{i=1}^{n}\nolimits\left(\lambda_{i}\cdot e_{i}\right)^{k_{i}}.
\end{align*}
Hence, $x$ has an expansion in terms of $E'$. To prove uniqueness,
assume that $x=\mu\prod_{i=1}^{n}\nolimits\lambda_{i}^{-\ell_{i}}\cdot\prod_{i=1}^{n}\nolimits\left(\lambda_{i}\cdot e_{i}\right)^{\ell_{i}}$.
Changing this expansion in terms of $E'$ to an expansion in terms
of $E''=\left\{ \lambda_{1}^{-1}\cdot\left(\lambda_{1}\cdot e_{1}\right),\ldots,\lambda_{n}^{-1}\cdot\left(\lambda_{n}\cdot e_{n}\right)\right\} $
gives 
\[
x=\mu\prod_{i=1}^{n}\nolimits\lambda_{i}^{-k_{i}}\prod_{i=1}^{n}\nolimits\lambda_{i}^{k_{i}}\cdot\prod_{i=1}^{n}\nolimits\left(\lambda_{i}^{-1}\cdot\left(\lambda_{i}\cdot e_{i}\right)\right)^{\ell_{i}}=\mu\cdot\prod_{i=1}^{n}\nolimits e_{i}^{\ell_{i}},
\]
so $\mu=\mu_{E}\left(x\right)$ and $\ell_{i}=k_{i}$ for $i=1,\ldots,n$
by the uniqueness of the expansion of $x$ in terms of $E=E''$.
\end{proof}
In general, the measure of a quantity thus depends on a choice of
basis, but %
there is an important exception to this rule.
\begin{prop}
\label{s3.4}Let $Q$ be a quantity space over $K$. For every $x\in\left[1_{Q}\right]$,
$\mu_{E}\left(x\right)$ does not depend on $E$.
\end{prop}

\begin{proof}
$1_{Q}$ is a unit quantity for $\left[1_{Q}\right]$ by Proposition
\ref{p22}, so there is a unique $\lambda\in K$ such that $x=\lambda\cdot1_{Q}$,
so $\mu_{E}\left(x\right)=\lambda\,\mu_{E}\left(1_{Q}\right)=\lambda$
for any basis $E$ for $Q$ by Propositions \ref{p25} and \ref{p38}(3). 
\end{proof}
\begin{rem}
\label{rem317}The $\pi$ theorem in dimensional analysis depends
on this result \cite{JON4}. It is common to refer to any $x\in\left[1_{Q}\right]$
as a ``dimensionless quantity'', although $x$ is not really dimensionless
\textendash{} it belongs to, or ``has'', the dimension $\left[1_{Q}\right]$.
Also, many authors (e.g., \cite{WAL,DROB,WHIT}) identify ''dimensionless
quantities'' with numbers,%
{} but Proposition \ref{s3.4} does not justify this identification.
A ''dimensionless quantity'' does not correspond to a unique number,
but to a number that depends on the choice of a quantity unit for
$\left[1_{Q}\right]$. For example, plane angles can be measured in
both radians and degrees. However, if we have a coherent system of
units $U$ then $\left[1_{Q}\right]$ contains exactly one unit $1_{Q}$
since $U$ is a submonoid of $Q$ and $u\sim1_{Q}$ implies $u=1_{Q}$.
Also, by Proposition \ref{p26-1} each choice of basis for $Q$ \textendash{}
that is, each choice of so-called base units \cite{VIM} \textendash{}
gives rise to a coherent system of units. Note that for a plane angle
$1_{Q}$ corresponds to the radian.%
\end{rem}

\subsection{\label{s26}$Q/{\sim}$ is a free abelian group}

In this section, we show that $Q/{\sim}$ regarded as a monoid has
additional properties derived from the quantity space $Q$. Below,
let $\breve{x}$ be given by $\breve{x}=1\cdot\prod_{i=1}e_{i}^{k_{i}}$,
where $x=\mu\cdot\prod_{i=1}e_{i}^{k_{i}}$ is the expansion of $x\in Q$
relative to a basis for $Q$. Note that, irrespective of the choice
of basis, $\breve{x}$ is a non-zero quantity by Proposition \ref{p20-1}
and such that $\breve{x}\sim x$ by Lemma \ref{s3.3}.
\begin{prop}
\label{p311}If $Q$ is a quantity space then $Q/{\sim}$ is an abelian
group.
\end{prop}

\begin{proof}
$Q/{\sim}$ is a commutative monoid since $\left[x\right]\left[y\right]=\left[xy\right]=\left[yx\right]=\left[y\right]\left[x\right]$
for all $\left[x\right],\left[y\right]\in Q/{\sim}$.%
{} Also, $\left[x\right]\left[\breve{x}^{-1}\right]=\left[x\breve{x}^{-1}\right]=\left[\mu\cdot1_{Q}\right]=\left[\breve{x}^{-1}x\right]=\left[\breve{x}^{-1}\right]\left[x\right]$,
so $\left[\breve{x}^{-1}\right]=\left[x\right]^{-1}$ since $\left[\mu\cdot1_{Q}\right]=\left[1_{Q}\right]=1_{Q/{\sim}}$.
Thus, $Q/{\sim}$ is an abelian group.
\end{proof}
Recall that a basis for a finitely generated abelian group $G$ is
a set $\left\{ \varepsilon_{1},\ldots,\varepsilon_{n}\right\} $ of
elements of $G$ such that every $x\in G$ has a unique expansion
$x=\prod_{i=1}^{n}\varepsilon_{i}^{k_{i}}$, where $k_{i}$ are integers.
\begin{prop}
\label{s3.7-1}Let $Q$ be a quantity space with a basis $E=\left\{ e_{1},\ldots,e_{n}\right\} $.
Then $\mathsf{E}=\left\{ \left[e_{1}\right],\ldots,\left[e_{n}\right]\right\} $
is a basis for $Q/{\sim}$ with the same cardinality as $E$. 
\end{prop}

\begin{proof}
The unique expansions of $e_{i},e_{j}\in E$ relative to $E$ are
\[
e_{i}=1\cdot\left(\cdots e_{i-1}^{0}e_{i}^{1}e_{i+1}^{0}\cdots\right),\quad e_{j}=1\cdot\left(\cdots e_{j-1}^{0}e_{j}^{1}e_{j+1}^{0}\cdots\right).
\]
Hence, if $e_{i}\neq e_{j}$ so that $i\neq j$ then $\left[e_{i}\right]\neq\left[e_{j}\right]$
by Lemma \ref{s3.3}. This means that the surjective mapping $\phi:E\rightarrow\left\{ \left[e_{1}\right],\ldots,\left[e_{n}\right]\right\} $
given by $\phi\left(e_{i}\right)=\left[e_{i}\right]$ is injective
as well and thus a bijection. It remains to show that $\mathsf{E}$
is a basis for $Q/{\sim}$.

First, let $\left[x\right]$ be an arbitrary dimension in $Q/{\sim}$.
As $E$ is a basis for $Q$, we have $x=\mu\cdot\prod_{i=1}^{n}e_{i}^{k_{i}}$
for some $\mu\in K$ and some integers $k_{1},\ldots,k_{n}$, so $\left[x\right]=\left[\mu\cdot\prod_{i=1}^{n}e_{i}^{k_{i}}\right]=\left[\prod_{i=1}^{n}e_{i}^{k_{i}}\right]=\prod_{i=1}^{n}\left[e_{i}\right]^{k_{i}}$.
Also, if $\left[x\right]=\prod_{i=1}^{n}\left[e_{i}\right]^{k_{i}}=\prod_{i=1}^{n}\left[e_{i}\right]^{\ell{}_{i}}$,
then $\left[\prod_{i=1}^{n}e_{i}^{k_{i}}\right]=\left[\prod_{i=1}^{n}e_{i}^{\ell{}_{i}}\right]$,
so $\prod_{i=1}^{n}e_{i}^{k_{i}}\sim\prod_{i=1}^{n}e_{i}^{\ell{}_{i}}$,
so $k_{i}=\ell_{i}$ for $i=1,\ldots,n$ by Lemma \ref{s3.3}.
\end{proof}
A (finitely generated) abelian group for which there exists a basis
is said to be free abelian (of finite rank). Hence, corresponding
to the fact that if $X$ is a scalable monoid then $X/{\sim}$ is
a monoid, we have the following much stronger result.
\begin{prop}
\label{314}If $Q$ is a quantity space then $Q/{\sim}$ is a free
abelian group of finite rank. 
\end{prop}

Recall that any two bases for a free abelian group $G$ have the same
cardinality, the rank of $G$. Proposition \ref{s3.7-1} thus leads
to an analogue of the dimension theorem for finite-dimensional vector
spaces.
\begin{prop}
\label{s3.8}If $Q$ is a quantity space then any two bases for $Q$
have the same cardinality. 
\end{prop}

\begin{proof}
If $E=\left\{ e_{1},\ldots,e_{n}\right\} $ and $E'=\left\{ e_{1}',\ldots,e_{m}'\right\} $
are bases for $Q$, so that $\mathsf{E}=\left\{ \left[e_{1}\right],\ldots,\left[e_{n}\right]\right\} $
and $\mathsf{E}'=\left\{ \left[e_{1}'\right],\ldots,\left[e_{m}'\right]\right\} $
are bases for $Q/{\sim}$, then $\left|E\right|=\left|\mathsf{E}\right|=\left|\mathsf{E'}\right|=\left|E'\right|$
by Proposition \ref{s3.7-1} and the equicardinality of bases for
a free abelian group.
\end{proof}
A quantity space with bases of cardinality $n$ is said to be of \emph{rank}
$n$.
\begin{example}
The dimensions corresponding to base quantities in the International
System of Units (SI) \cite{SI}, such as the dimensions of length,
time and mass, denoted $\mathsf{L}$, $\mathsf{T}$ and $\mathsf{M}$,
respectively, are elements of a basis for some free abelian group
$Q/\!\sim$. For example, $\left\{ \mathsf{L},\mathsf{T},\mathsf{M}\right\} $
is a basis for $Q/\!\sim$, where $Q$ is a quantity space for classical
mechanics. This is not the only possible basis, however. For example,
$\left\{ \mathsf{L},\mathsf{T},\mathsf{F}\right\} ,$ where $\mathsf{F}=\mathsf{MLT}^{-2}$,
is another three-element basis for $Q/\!\sim$, and another possible
set of base dimensions for classical mechanics.
\end{example}

Let us consider quantity spaces $Q$ and $Q'$ over $K$ with bases
$E=\left\{ e_{1},\ldots,e_{n}\right\} $ and $E'=\left\{ e_{1}',\ldots,e_{n}'\right\} $.
It is easy to verify that a bijection $\phi:E\rightarrow E'$ can
be extended to an isomorphism $\phi^{*}:Q\rightarrow Q'$ by setting
$\phi^{*}\left(\mu\cdot\prod_{i=1}^{n}e_{i}^{k_{i}}\right)=\mu\cdot\prod_{i=1}^{n}\phi\left(e_{i}\right)^{k_{i}}$.
Conversely, if $\phi^{*}:Q\rightarrow Q'$ is an isomorphism then
$\left\{ \phi^{*}\left(e_{1}\right),\ldots,\phi^{*}\left(e_{n}\right)\right\} $
is clearly a basis for $Q'$ of the same cardinality as $E$. These
observations lead to the following classification theorem, similar
to a theorem in linear algebra:
\begin{prop}
Quantity spaces over the same field are isomorphic if and only if
they are of the same rank (cf. \cite{RAP}).
\end{prop}

There is a reciprocal connection between bases for $Q$ and bases
for $Q/\!\sim$.
\begin{prop}
\label{p33}Let $Q$ be a quantity space, and let $\mathsf{E}=\left\{ \mathsf{e}_{1},\ldots,\mathsf{e}_{n}\right\} $
be a basis for $Q/{\sim}$. Then there is a subset $E=\left\{ e_{1},\ldots,e_{n}\right\} $
of $Q$ such that $e_{i}\in\mathsf{e}_{i}$ and $E$ is a basis for
$Q$.
\end{prop}

\begin{proof}
We can choose a function $\psi:\mathsf{E}\rightarrow\left\{ \psi\left(\mathsf{e}{}_{1}\right),\ldots,\psi\left(\mathsf{e}{}_{n}\right)\right\} $
such that we have $0_{\mathsf{e}_{i}}\neq\psi\left(\mathsf{e}{}_{i}\right)\in\mathsf{e}_{i}$
for all $\mathsf{e}_{i}$. This is a surjective function, and $\mathsf{e}_{i}\neq\mathsf{e}_{j}$
implies $\mathsf{e}_{i}\cap\mathsf{e}_{j}=\textrm{Ø}$, so  $\psi$
is injective as well and hence a bijection. For convenience, we write
$\psi\left(\mathsf{e}{}_{i}\right)$ as $e_{i}$. Each $e_{i}$ is
invertible by Proposition \ref{p20-1}.

Let $x$ be an arbitrary quantity in $Q$. As $\mathsf{E}$ is a basis
for $Q/{\sim}$, we have $\left[x\right]=\prod_{i=1}^{n}\mathsf{e}_{i}^{k_{i}}=\prod_{i=1}^{n}\left[e_{i}\right]^{k_{i}}=\left[\prod_{i=1}^{n}e_{i}^{k_{i}}\right]$
for some integers $k_{1},\ldots,k_{n}$, and as $e_{i}\neq0_{\mathsf{e}_{i}}$
for each $e_{i}$, $\prod_{i=1}^{n}e_{i}^{k_{i}}$ is non-zero and
thus a unit quantity for $\left[x\right]$ by Proposition \ref{p22}.
Hence, there exists a unique $\mu\in K$ for $\prod_{i=1}^{n}e_{i}^{k_{i}}$
such that $x=\mu\cdot\prod_{i=1}^{n}e_{i}^{k_{i}}$. Also, if $x=\mu\cdot\prod_{i=1}^{n}e_{i}^{k_{i}}=\nu\cdot\prod_{i=1}^{n}e_{i}^{\ell_{i}}$
then $\left[\prod_{i=1}^{n}e_{i}^{k_{i}}\right]=\left[\prod_{i=1}^{n}e_{i}^{\ell_{i}}\right]$,
so $\prod_{i=1}^{n}\left[e_{i}\right]^{k_{i}}=\prod_{i=1}^{n}\left[e_{i}\right]^{\ell_{i}}$,
so $\ell_{i}=k{}_{i}$ for $i=1,\ldots,n$, since $\mathsf{E}$ is
a basis for $Q/{\sim}$, so $\nu=\mu$.
\end{proof}
We can now extend to quantity spaces the theorem that a subgroup of
a free abelian group is free abelian, using this fact.
\begin{prop}
If a subalgebra $Q'$ of a quantity space $Q$ regarded as a scalable
monoid contains all inverses of elements of $Q'$ then $Q'$ is a
quantity space.
\end{prop}

\begin{proof}
First note that $Q'$ is a scalable monoid, so $Q'\!/{\sim}$ is a
monoid. Also, recall from the proof of Proposition \ref{p311} that
$\left[\breve{x}^{-1}\right]=\left[x\right]^{-1}$ so if $\breve{x}\in Q'$
implies $\breve{x}^{-1}\in Q'$ then $\left[x\right]\in Q'/{\sim}$
implies $\left[x\right]^{-1}\in Q'/{\sim}$ since $x\in Q'$ implies
$\breve{x}\in Q'$. Hence, $Q'\!/{\sim}$ is a subgroup of $Q/{\sim}$,
so $Q'\!/{\sim}$ is a free abelian group with a basis $\mathsf{E}$
corresponding to a basis $E$ for $Q'$ by Proposition \ref{p33}.
\end{proof}
This result is analogous also to the simple fact that a submodule
of a vector space is a vector space, so we have found yet another
similarity between free abelian groups, quantity spaces and vector
spaces.%

\appendix

\section{Coherent systems of units in quotient spaces}
\begin{prop}
\label{p9-1}Let $X$ be a scalable monoid over $R$, $S$ a coherent
system of unit elements for $X$, and $T\subseteq S$ a normal submonoid
of $X$. Then $X/T$ is a scalable monoid, $\left[t\right]_{T}=\left[1_{\!X}\right]_{T}$
for any $t\in T$, and $S/T=\left\{ \left[s\right]_{T}\mid s\in S\right\} $
is a coherent system of unit elements for $X/T$. 
\end{prop}

\begin{proof}
By Proposition \ref{p8}, $X/T$ is a scalable monoid since $T$ is
a normal submonoid of $X$, and if $t\in T$ then $t\sim_{_{T}}\!1_{\!X}$
since $1_{\!X}t=t1_{\!X}$ and $1_{\!X}\in T$.

Thus, $\left[t\right]_{T}=\left[1_{\!X}\right]_{T}=1_{X/T}$ since
\[
\left[1_{\!X}\right]_{T}\left[x\right]_{T}=\left[1_{\!X}x\right]_{T}=\left[x\right]_{T}=\left[x1_{\!X}\right]_{T}=\left[x\right]_{T}\left[1_{\!X}\right]_{T}
\]
 for any $x\in X$. Also, if $\left[s\right]_{T},\left[s'\right]_{T}\in S/T$,
meaning that $s,s'\in S$, then $\left[s\right]_{T}\left[s'\right]_{T}=\left[ss'\right]_{T}\in S/T$
since $ss'\in S$. Hence, $S/T$ is a submonoid of $X/T$.

Assume that $\left[s\right]_{T}\sim\left[s'\right]_{T}$, where $s,s'\in S$.
Then $\rho\cdot\left[s\right]_{T}=\sigma\cdot\left[s'\right]_{T}$
for some $\rho,\sigma\in R$, so $\left[\rho\cdot s\right]_{T}=\left[\sigma\cdot s'\right]_{T}$,
so $\rho\cdot s\sim_{T}\sigma\cdot s'$, so $t\left(\rho\cdot s\right)=t'\left(\sigma\cdot s'\right)$
for some $t,t'\in T$, so $\rho\cdot ts=\sigma\cdot t's'$, so $ts\sim t's'$
where $ts,t's'\in S$. Hence, $ts=t's'$ since $S$ is sparse in $X$,
so $s\sim_{T}s'$, %
meaning that $\left[s\right]_{T}=\left[s'\right]_{T}$. Thus, $S/T$
is a sparse set of elements of $X/T$.

Consider any $\mathsf{C}\in X/T$ and let $x\in X$ be such that $\left[x\right]_{T}\in\mathsf{C}$.
Then there is some $s\in S$ and some $\rho\in R$ such that $x=\rho\cdot s$
since $S$ is a dense set of unit elements in $X$, and hence $\left[x\right]_{T}=\left[\rho\cdot s\right]_{T}=\rho\cdot\left[s\right]_{T}$.
If $\rho\cdot\left[s\right]_{T}=\sigma\cdot\left[s\right]_{T}$ then
$\left[\rho\cdot s\right]_{T}=\left[\sigma\cdot s\right]_{T}$, so
$\rho\cdot s\sim_{T}\sigma\cdot s$, and as in the preceding paragraph
this implies that %
$\rho\cdot ts=\sigma\cdot t's$ and $ts=t's$ for some $t,t'\in T$%
. Hence, $\rho=\sigma$ since $ts\in S$ is a unit element for $\left[ts\right]$.
This means that $\left[s\right]_{T}$ is a unit element for $\mathsf{\mathsf{C}},$
so $S/T$ is dense in $X/T$ as well as sparse.%
\end{proof}
As a simple example, $S=\left\{ 1x^{k_{1}}y^{k_{2}}z^{k_{3}}\mid k_{1},k_{2},k_{3}\in\mathbb{Z}_{\geq0}\right\} $
is a coherent system of unit elements for $\mathbb{R}\left\llbracket x;y;z\right\rrbracket $,
$T=\left\{ 1x^{k_{1}}y^{k_{2}}z^{0}\mid k_{1},k_{2}\in\mathbb{Z}_{\geq0}\right\} $
is a normal submonoid of $\mathbb{R}\left\llbracket x;y;z\right\rrbracket $,
$T\subseteq S$, and $\mathbb{R}\left\llbracket x;y;z\right\rrbracket /T$
is isomorphic to $\mathbb{R}\left\llbracket z\right\rrbracket $.

A typical application of Proposition \ref{p9-1} in physics is described
by Raposo \cite{RAP}:
\begin{quotation}
{\footnotesize{}The mechanism of taking quotients is the algebraic
tool underlying what is common practice in physics of choosing \textquotedblleft systems
of units'' such that some specified universal constants become dimensionless
and take on the numerical value 1. {[}...{]} But it has to be remarked
that the mechanism goes beyond a change of system of units; it is
indeed a change of space of quantities. {[}p. 153{]}}{\footnotesize \par}
\end{quotation}
For example, one may decide to measure both time and length by means
of a measure for length, using the universal constant $c$, thus introducing
a new system of units such that times and lengths are not distinguishable
(see further \cite{KIT}). Such an operation amounts to a projection
$x\mapsto\left[x\right]_{T}$ of the original space of quantities
$X$ onto a quotient space $X/T$. In terms of $S$ and $T$, if $S=\left\{ t^{k_{1}}\ell^{k_{2}}\mid k_{1},k_{2}\in\mathbb{Z}\right\} $,
where $t$ is a unit for time and $\ell$ a unit for length, is a
coherent system of unit elements for $X$ and we set $T=\left\{ t^{k}\mid k\in\mathbb{Z}\right\} $
then $S/T$ corresponds to $\left\{ \ell^{k}\mid k\in\mathbb{Z}\right\} $,
meaning that both time and length are measured by reference to a unit
in $X/T$ corresponding to a unit for length in $X$.

\newpage{}

\section{Notes on related work and fields of application}

\subsection{\label{a1}Theoretical approaches}

In 1945, Landolt \cite{LAN}, apparently%
{} inspired by the development of abstract algebra during the interwar
period in Europe \cite{vdWAE}, called attention to group operations
on systems of quantities. Specifically, he pointed out that the invertible
quantities form a group under \textquotedblleft \emph{qualitative
Verknüpfung}\textquotedblright , that is, multiplication of quantities,
and that quantities of the same kind form a group under \textquotedblleft \emph{intensive
Verknüpfung}\textquotedblright , that is, addition of quantities. 

In a seminal article, Fleischmann \cite{FLEI} shifted the focus from
quantities (\emph{Grössen}) to \emph{kinds} of quantities (\emph{Grössenarten}),
and suggested that kinds of quantities can themselves be multiplied,
requiring that if $q$ is of kind $K$ and $q'$ is of kind $\text{ \ensuremath{K'} }$
then $qq'$ is of kind $KK'$. He also proposed that a set of kinds
of quantities with a product defined in this way would be a finitely
generated free abelian group. %

Quade \cite{QUAD} defines systems of quantities by means of a rather
complicated construction with one-dimensional vector spaces as building
blocks. As a first step, he defines a quantity system as the union
$U$ of all vector spaces $V_{i}$ in a countably infinite set $V$
of pairwise disjoint one-dimensional vector spaces over the real or
complex numbers. He then assumes that for any $x,x'\in U$ there exists
an associative, commutative product $xx'$ satisfying $\lambda\left(xx'\right)=\left(\lambda x\right)x'$
and such that $VV'=\left\{ xx'\mid x\in V,x'\in V'\right\} $ is a
one-dimensional vector space. Next, he selects a finite number of
vector spaces $V_{1},\ldots,V_{n}\in V$ and considers the set $\mathfrak{G}$
of all products of vector spaces of the form $\prod\nolimits _{j=1}^{n}V_{j}^{k_{j}}$,
where $k_{j}$ are integers. $\mathfrak{G}$ is a finitely generated
free abelian group of rank $n$. To ensure the supply of inverses
of quantities, he embeds the set of non-zero elements of the selected
vector spaces in a group of fractions $\mathfrak{S}$, using the fact
that $U$ is a commutative semigroup.%

Quade's construction is actually even more complicated than sketched
here. The point is that $\mathfrak{S}$ corresponds to a set of quantities,
while $\mathfrak{G}$ corresponds to a set of kinds of quantities
(or dimensions). Landolts and Fleischmann's ideas are elaborated formally,
but unfortunately not clarified.

{} 

Carlson's \cite{CARL} definition of quantities is based on a set
of ''pre-units'' which is in effect a predefined basis $\Xi$ for
a finite-dimensional vector space $V_{\Xi}$ over $\mathbb{Q}$ with
multiplication as binary operation, so that $\xi+\eta$ is written
as $\xi\eta$ and $\lambda\xi$ as $\xi^{\lambda}$. Carlson then
defines a quantity as a pair $\left(r,\xi\right)$, where $r$ is
a real number and $\xi$ a pre-unit. Multiplication of quantities
by real numbers is defined by setting $a\left(r,\xi\right)=\left(ar,\xi\right)$,
multiplication of quantities is defined by setting $\left(r,\xi\right)\left(s,\eta\right)=\left(rs,\xi\eta\right)$
and fractional powers of quantities are defined by setting $\left(r,\xi\right)^{m/n}=\left(\sqrt[n]{r^{m}},\xi^{m/n}\right)$
if a (unique) positive real nth root $\sqrt[n]{r^{m}}$ of $r^{m}$
exists; we obtain a quantity structure $\mathbb{R}\times V_{\Xi}$.
A ''fundamental system of units'' for $\mathbb{R}\times V_{\Xi}$
is a set $\left\{ \left(u_{1},\xi_{1}\right),\ldots,\left(u_{m},\xi_{m}\right)\right\} $
such that $u_{i}\in\mathbb{R}_{>0}$ and $\xi_{i}\in\Xi$, where every
$q\in\mathbb{R}\times V_{\Xi}$ has a unique expansion $q=r\left(u_{1},\xi_{1}\right)^{\epsilon_{1}}\cdots\left(u_{m},\xi_{m}\right)^{\epsilon_{m}}$,
where $r\in\mathbb{R}$. The sequence of exponents $\epsilon_{1},\ldots,\epsilon_{m}$
associated with $q$, which is clearly the same for all fundamental
systems of units, Carlson calls the ''dimensions'' of $q$, and
he says that quantities are of the same kind if they have the same
dimensions.

Carlson's construction is incomplete in the sense that vector space
operations for quantities of the same kind are not considered, and
multiplication of kinds of quantities is not defined. Instead, we
have multiplication of pre-units, which can be seen as units in a
fixed, coherent system of units.

{} 

The approach introduced by Drobot \cite{DROB} and developed by Whitney
\cite{WHIT} is based on the idea that the set of quantities itself
\textendash{} rather than a set of pre-units \textendash{} is just
a vector space $V_{Q}$ under multiplication of quantities and over
a field $\mathbf{R}$%
, so there is a scalar product $q^{\lambda}$, where $\lambda\in\mathbf{R},q\in V_{Q}$.
$V_{Q}$ is also assumed to contain a set $\boldsymbol{R}$ of scalars,
so another scalar product can be defined as a usual product $rq$,
where $r\in\boldsymbol{R},q\in V_{Q}$. (Thus, the authors identify
dimensionless quantities with scalars.) %
{} Both Drobot and Whitney define, in slightly different ways, sets
of quantities of the same kind, called \textquotedblleft dimensions\textquotedblright{}
by Drobot and \textquotedblleft birays\textquotedblright{} by Whitney,
and both define addition of quantities of the same kind, $q=\alpha u$
and $r=\beta u$ where $\alpha,\beta$ are scalars and $u$ is non-zero,
by the identity $q+r=\left(\alpha+\beta\right)\cdot u$ (although
only Whitney proves that this definition is legitimate). Letting $\left[q\right]$
denote the dimension/biray containing $q$, both authors define multiplication
of dimensions/birays by the identity $\left[x\right]$$\left[y\right]$
= $\left[xy\right]$, and exponentiation by $\left[x\right]^{\lambda}=\left[x^{\lambda}\right]$.
Whitney also proves that $q\left(r+s\right)=qr+qs$ for any quantities
$q,r,s$ such that $\left[r\right]=\left[s\right]$ and that $qs=rs$,
where $s$ is non-zero, implies $q=r$.%
{} 

Unfortunately, the assumptions that $V_{Q}$ is a vector space over
$\mathbf{R}$ with scalar multiplication $\left(\lambda,q\right)\mapsto q^{\lambda}$
and that $V_{Q}$ contains a set of scalars $\boldsymbol{R}$ are
not fully compatible. In particular, $q^{\lambda}$ is not a real
number for all $q,\lambda\in\mathbb{R}$, and if we require that $\boldsymbol{R}\subseteq\mathbb{R}_{>0}$
to avoid this problem then all quantities must be positive. Anyway,
while integral powers of quantities make sense in physics, it is not
clear how to interpret $q^{0.2}$ or $q^{\pi}$, where $q$ is a ''dimensionful''
quantity rather than a number.

Kock \cite{KOCK} proposed a limited but elegant construction also
based on the 'vector-space-with-embedded-scalars' idea. Accepting
the restriction to positive quantities, he pointed out that in the
short exact sequence of vector spaces over $\mathbb{Q}$,
\[
\mathbb{Q}_{>0}\overset{\iota}{\longrightarrow}P\overset{d}{\longrightarrow}D,
\]
where $\iota$ is an inclusion map, $d$ a surjective $\mathbb{Q}$-linear
map, and the kernel of $d$ is the image of $\iota$, $P$ can be
interpreted as a set of quantities, $D$ as a set of dimensions, $Q_{>0}$
as the ''dimensionless'' quantities in $P$ and, for every $M\in D$,
$d^{-1}\left(M\right)$ as a set of quantities of the same kind. The
operations on $\mathbb{Q}_{>0}$, $P$ and $D$ are the usual operations
in (multiplicatively written) vector spaces over $\mathbb{Q}$, and
the identities in the algebraic structure containing $\mathbb{Q}_{>0}$,
$P$ and $D$ are those that follow from the%
{} vector space axioms and the $\mathbb{Q}$-linearity of $\iota$ and
$d$.

More recently, Raposo \cite{RAP} has proposed a definition of a system
or \textquotedblright space\textquotedblright{} of quantities somewhat
similar to Quade\textquoteright s but more concise and elegant. By
this definition, a space of quantities $Q$ is an algebraic fiber
bundle, with fibers of quantities attached to dimensions (kinds of
quantities) in a base space assumed to be a finitely generated free
abelian group. Each fiber is again a one-dimensional vector space,
with scalar product $\lambda q$. Multiplication of quantities and
multiplication of dimensions are defined independently, but are assumed
to be compatible in the same sense as for Quade. The quantities constitute
a commutative monoid, and it is assumed that $q\left(r+s\right)=qr+qs$
for any quantities $q,r,s$, where $r,s$ are of the same kind, and
that $\lambda(qr)=(\lambda q)r$ for any scalar $\text{\ensuremath{\lambda}}$
and quantities $q,r$. Although this is technically not part of the
definition of a space of quantities, Raposo also assumes that if $q$
and $r$ are non-zero quantities then $qr$ is a non-zero quantity. 

Raposo's theory of quantity spaces (supplemented with a 'no zero divisors'
condition) is complete and free from anomalies. Compared to the theory
presented in this article, it contains some redundant elements, since
there are more primitive notions. However, the two theories have been
shown to be completely equivalent \cite{RAP2}. This lends credence
to both theories, since they were developed independently.

The table below contains a simplified comparison of some aspects of
six of the approaches to quantity calculus reviewed above.\medskip{}
\begin{singlespace}
\noindent \begin{center}
\begin{tabular}{>{\centering}m{2.5cm}>{\centering}m{1.2cm}>{\centering}m{1.2cm}>{\centering}m{1.2cm}>{\centering}m{1.2cm}>{\centering}m{1.2cm}>{\centering}m{1.2cm}}
\toprule 
\emph{Aspects} & \multicolumn{6}{c}{\emph{Authors}}\tabularnewline
\cmidrule{2-7} 
 & {\footnotesize{}Drobot (1953)} & {\footnotesize{}Quade (1961)} & {\footnotesize{}Whitney (1968)} & {\footnotesize{}Carlson (1979)} & {\footnotesize{}Jonsson (2014)} & {\footnotesize{}Raposo (2016)}\tabularnewline
\midrule
{\footnotesize{}Addition of quantities} & {\scriptsize{}Derived} & {\scriptsize{}Primitive} & {\scriptsize{}Derived} & {\scriptsize{}None} & {\scriptsize{}Derived} & {\scriptsize{}Primitive}\tabularnewline
{\footnotesize{}Scalar product }\\
{\footnotesize{}of quantities} & {\scriptsize{}Derived} & {\scriptsize{}Primitive} & {\scriptsize{}Derived} & {\scriptsize{}Derived} & {\scriptsize{}Primitive} & {\scriptsize{}Primitive}\tabularnewline
{\footnotesize{}Product of quantities} & {\scriptsize{}Primitive} & {\scriptsize{}Primitive} & {\scriptsize{}Primitive} & {\scriptsize{}Derived} & {\scriptsize{}Primitive} & {\scriptsize{}Primitive}\tabularnewline
{\footnotesize{}Product of dimensions } & {\scriptsize{}Derived} & {\scriptsize{}Derived} & {\scriptsize{}Derived} & {\scriptsize{}(Primitive)} & {\scriptsize{}Derived} & {\scriptsize{}Primitive}\tabularnewline
{\footnotesize{}Exponent of quantities} & {\scriptsize{}Primitive} & {\scriptsize{}Derived} & {\scriptsize{}Primitive} & {\scriptsize{}Derived} & {\scriptsize{}Derived} & {\scriptsize{}Derived}\tabularnewline
{\footnotesize{}Exponent of dimensions } & {\scriptsize{}Derived} & {\scriptsize{}Derived} & {\scriptsize{}Derived} & {\scriptsize{}(Primitive)} & {\scriptsize{}Derived} & {\scriptsize{}Derived}\tabularnewline
{\footnotesize{}Ring of exponents } & {\scriptsize{}$\mathbb{R}$} & {\scriptsize{}$\mathbb{Z}$} & {\scriptsize{}$\mathbb{Q}$ or $\mathbb{R}$} & {\scriptsize{}$\mathbb{Q}$} & {\scriptsize{}$\mathbb{Z}$} & {\scriptsize{}$\mathbb{Z}$}\tabularnewline
\bottomrule
\end{tabular} 
\par\end{center}
\end{singlespace}

{\scriptsize{}Note: For Carlson, we consider products and exponents
of pre-units instead of dimensions.}{\scriptsize \par}

\subsection{\label{a2}Rules of quantity calculus}

In \cite{BOER}, de Boer lists some fundamental rules for calculation
with quantities, drawn from the literature on quantity calculus. All
these rules can be derived from the theory of quantity spaces. The
list below includes references to de Boer's rules and relevant definitions
or results from this article.
\begin{enumerate}
\item (A2.1; Definition \ref{thm:def1}, Corollary \ref{cor35}). A quantity
can be multiplied by a quantity as in a monoid, and the non-zero quantities
form an abelian group under multiplication. 
\item (A2.2; Definition \ref{thm:def1}). A quantity $q$ can be multiplied
by a number $\lambda$, and we have the identities $1\cdot q=q$ and
$\alpha\cdot\left(\beta\cdot q\right)=\left(\alpha\beta\right)\cdot q$. 
\item (A2.2; Definition \ref{thm:def1}). Multiplication of quantities by
numbers and by quantities are related by $\lambda\cdot\left(pq\right)=\left(\lambda\cdot p\right)q=p\left(\lambda\cdot q\right)$.
\item (A3.1, A6.1; Definition \ref{d3.1}, Proposition \ref{s3.1}). A system
of quantities can be partitioned into equivalence classes of quantities
of the same kind. (de Boer makes a distinction between such equivalence
classes and dimensions, but in the present theory these two notions
coincide.)
\item (A3.2; Proposition \ref{p38-1}). Quantities of the same kind can
be added and form an abelian group under addition. 
\item (A3.3; Proposition \ref{p38-1}). If $p$ and $q$ are of the same
kind then $\lambda\cdot\left(p+q\right)=\lambda\cdot p+\lambda\cdot q$
and $\left(\alpha+\beta\right)\cdot q=\alpha\cdot q+\beta\cdot q$. 
\item (A3.4; Remark \ref{rem317}.) All ''dimensionless'' quantities are
of the same kind.
\item (A4.1, A6.2; Proposition \ref{314}). Kinds of quantities can be multiplied,
forming a finitely generated free abelian group under multiplication. 
\item (A4.2, A6.2; Proposition \ref{s3.2}, Definition \ref{def:214}).
If $q$ is a quantity of kind $K$ and $q'$ a quantity of kind $K'$
then $qq'$ is a quantity of kind $KK'$. 
\item (A5.1; Proposition \ref{preP39}). For every kind of quantities and
every quantity $q$ of this kind there is a quantity $u$ of the same
kind such that $q=\mu\cdot u$, where $\mu$ is a uniquely determined
number. Such a quantity $u$ is called a unit. 
\item (A5.1, A5.2; Proposition \ref{p26-1}). It is possible to select exactly
one unit from each kind of quantities in such a way that that if $u$
is the unit of kind $K$ and $u'$ is the unit of kind $K'$ then
$uu'$ is the unit of kind $KK'$. A set of units satisfying this
condition is said to be coherent.%
\end{enumerate}
We can derive some more fundamental rules not considered by de Boer
in \cite{BOER}.
\begin{enumerate}
\item[(i)] (Corollary \ref{c31}). $\lambda\cdot q$ is a quantity of the same
kind as $q$. 
\item[(ii)] (Proposition \ref{p310}). If $q$ is a quantity and $r,s$ are quantities
of the same kind then $q\left(r+s\right)=qr+qs$. 
\item[(iii)] (Corollary \ref{cor35}.) If $q$ and $r$ are non-zero quantities
then $qr$ is non-zero.
\end{enumerate}

\subsection{\label{a3}On applications of the theory of quantity spaces}

At the heart of theoretical metrology is the relationship between
measures and what is measured \textendash{} quantities. Quantity space
theory can clarify this relationship by elucidating the nature of
quantities and systems of quantities.

In the \emph{International Vocabulary of Metrology} from 2012 (VIM3)
\cite{VIM}, one sense of ''quantity'' (1.1) is a generic one, corresponding
to \emph{Grössenart} or kind of quantity, while a ''quantity value''
(1.19) represents a particular \emph{Grösse} or a quantity as defined
here. Unit quantities are called ''(measurement) units'' (1.9) in
VIM3. A set of ''base quantities'' (1.4) is in effect a set $\mathsf{E}$
of selected kinds of quantities, or equivalently dimensions, which
is a basis for some $Q/{\sim}$; the elements of a corresponding basis
$E=\left\{ e_{1},\ldots,e_{n}\right\} $ for $Q$ are ''base units''
(1.10). Further, in VIM3 a ''derived unit'' (1.11) is some $u\in Q$,
other than a base unit, with an expansion $u=\mu\cdot\prod_{i=1}^{n}e_{i}^{k_{i}}$,
where $\mu\neq0$, while a ''coherent derived unit'' (1.12) is a
derived unit $v$ with an expansion $v=1\cdot\prod_{i=1}^{n}e_{i}^{k_{i}}$.
Following Fourier \cite{FOUR}, VIM3 defines a ''quantity dimension''
(1.7) of a quantity $q\in Q$ in terms of an integer tuple $\left(k_{1},\ldots,k_{n}\right)$
describing how $q$ is expressed as $\mu_{E}\left(q\right)\cdot\prod_{i=1}^{n}e_{i}^{k_{i}}$,
where $\left\{ e_{1},\ldots,e{}_{n}\right\} $ is a quantity-space
basis for $Q$. 

It would be of interest to fully analyze VIM3 (or the upcoming VIM4)
in the light of the theory of quantity spaces%
. The distinction between concrete and abstract quantities \cite[p. 8--12]{JON3}
should be taken into account in this connection,

The relationship between measures and the quantities that they represent
is fundamental also in dimensional analysis, which is based on a principle
of covariance: \foreignlanguage{british}{a relation between scalars
representing a relation between quantities relative to a system of
unit quantities must continue to hold when that system is changed
in a legitimate way, although individual scalars may change as unit
quantities change. Reference \cite{JON4} presents an approach to
dimensional analysis explicitly based on this principle and expressed
in terms of quantities, dimensions and quantity functions rather than
scalars, units and scalar functions.}

A measure of a quantity is usually assumed to be a real number, but
in a quantity space over $K$ a measure is an element of $K$, where
$K$ is any field. For example, a measure can be a complex number.
This makes the present theory of quantity spaces well suited for applications
to problems of quantum physics.%

\end{document}